\documentclass[journal,10pt]{IEEEtran}
\usepackage[colorlinks,linkcolor=black,anchorcolor=black,linktocpage=true,urlcolor=blue,citecolor=black]{hyperref}

\usepackage{stfloats}
\usepackage{booktabs}
\usepackage{algorithmic}
\usepackage{tikz}
\usetikzlibrary{arrows}
\usepackage{array}
\usepackage{graphicx}
\usepackage{url}
\usepackage{cite}
\usepackage{amsmath,amsthm,amsfonts,amssymb,mathrsfs,bm}
\allowdisplaybreaks[4]
\usepackage{tabularx,array,enumerate,pifont,tikz,epstopdf}
\usetikzlibrary{shapes}

\makeatletter
\newcommand{\biggg}{\bBigg@{3}}

\newcommand{\Biggg}{\bBigg@{3.5}}

\newcommand{\bigggg}{\bBigg@{5}}

\def\biggggr{\mathclose\bigggg}
\newcommand{\Bigggg}{\bBigg@{6.5}}

\def\Biggggr{\mathclose\Bigggg}
\makeatother

\newtheorem{theorem}{Theorem}
\newtheorem{lemma}{Lemma}

\newtheorem{remark}{Remark}

\newtheorem{assumption}{Assumption}

\newtheorem{problem}{Problem}
\begin{document}

\title{\hspace*{-2.8mm} \LARGE \bf Optimal Output Consensus of Second-Order Uncertain Nonlinear Systems on Weight-Unbalanced Directed Networks}

\author{Jin~Zhang,
        Lu~Liu,~\IEEEmembership{Senior Member, IEEE,}
        and~Haibo~Ji
\thanks{J. Zhang is with the Department of Automation, University of Science and Technology of China, Hefei 230027, China, and also with the Department of Biomedical Engineering, City University of Hong Kong, Hong Kong (e-mail: zj55555@mail.ustc.edu.cn).
	
	L. Liu is with the Department of Biomedical Engineering, City University of Hong Kong, Hong Kong (e-mail: luliu45@cityu.edu.hk).
	
	H. Ji is with the Department of Automation, University of Science and Technology of China, Hefei 230027, China (e-mail: jihb@ustc.edu.cn).}}

\maketitle

\begin{abstract}
	This paper investigates the distributed optimal output consensus problem of second-order uncertain nonlinear multi-agent systems over weight-unbalanced directed networks. Under the standard assumption that local cost functions are strongly convex with globally Lipschitz gradients, a novel distributed dynamic state feedback controller is developed such that the outputs of all the agents reach the optimal solution to minimize the global cost function which is the sum of all the local cost functions. The controller design is based on a two-layer strategy, where a distributed optimal coordinator and a reference-tracking controller are proposed to address the challenges arising from unbalanced directed networks and uncertain nonlinear functions respectively. A key feature of the proposed controller is that the nonlinear functions containing the uncertainties and disturbances are not required to be globally Lipschitz. Furthermore, by exploiting adaptive control technique, no prior knowledge of the uncertainties or disturbances is required either. Two simulation examples are finally provided to illustrate the effectiveness of the proposed control scheme.
\end{abstract}


\IEEEpeerreviewmaketitle

\section{Introduction}

\IEEEPARstart{I}{n} the past decade, the distributed optimization problem (DOP) has experienced significant advance due to its wide applications in a broad range of areas, including power systems, resource allocation and sensor networks, see \cite{Yang2019survey,Molzahn2017survey,Ram2010distributed}. The typical DOP aims at driving all the agents in a distributed manner towards the optimal solution of a global cost function which is often defined to be the sum of all the local cost functions attached to individual agents. Many existing works on this topic primarily focus on discrete-time cases, see, for example, \cite{Nedic2009distributed,Nedic2014distributed,Xi2017dextra,Xi2016distributed,Xi2018linear} and references therein. More recently, much effort has been devoted to distributed continuous-time optimization problems \cite{Wang2010control,Gharesifard2013distributed,Kia2015distributed,Tran2018distributed,zhang2017distributed,Yi2018distributed,Xie2019global,Tang2018optimal,li2019distributed}. A plausible reason is that many practical systems operate in a continuous-time setting, such as unmanned vehicles and robots among others \cite{Yang2019survey}. 

It is worth noting that the conventional DOP in the aforementioned works can be reformulated as a distributed optimal output consensus (OOC) problem of multi-agent systems with single integrator agent dynamics. However, several engineering scenarios in practice could be formulated as the OOC problem of multi-agent systems with more general agent dynamics, such as the economic dispatch in power systems \cite{stegink2016Unifying}, rigid body attitude formation control \cite{song2017Relative} and source seeking in multi-robot systems \cite{zhang2011extremum}. Some works have been reported recently in solving the OOC problem of multi-agent systems with double integrators agent dynamics \cite{Tran2018distributed,zhang2017distributed,Yi2018distributed} and high-order linear agent dynamics \cite{Xie2019global,Tang2018optimal,li2019distributed} over undirected graphs. 

More recently, several works on distributed optimization of nonlinear multi-agent systems are reported \cite{Wang2015distributed,liu2019distributed,Tang2018distributed,li2018consensus,li2021distributed,li2020distributed,tang2020optimal}. The authors in \cite{Wang2015distributed} propose two distributed controllers to solve the OOC problem for output feedback nonlinear systems over undirected graphs, while the authors in \cite{liu2019distributed} address the same problem for the same nonlinear systems but over balanced directed graphs. It should be pointed out that the nonlinear dynamics considered in \cite{Wang2015distributed,liu2019distributed} contain external disturbances, and an internal model is designed to address this challenge. Later, based on a two-layer control strategy, the authors in \cite{Tang2018distributed} develop an adaptive controller to deal with the difficulty resulting from unknown nonlinear dynamics. The developed controller consists of an optimal coordinator that generates the optimal solution and a reference-tracking controller that ensures each agent follows its private optimal coordinator. However, the proposed controller can only be applied when the unknown nonlinear dynamics are linearly parameterized, and thus greatly limits the scope of its application. The authors in \cite{li2018consensus} and \cite{li2021distributed} investigate the distributed optimization problem and resource allocation problem of multi-agent systems with second-order nonlinear dynamics respectively, where the nonlinear functions are unfortunately required to be globally Lipschitz. The distributed optimization problem of multi-agent systems with more general nonlinear dynamics in normal form are considered in \cite{tang2020optimal,li2020distributed}.

It is noted that the network topologies in the above-mentioned works are limited to undirected graphs or balanced digraphs. However, since the information exchange between agents may be unidirectional due to limited bandwidth, it is thus of both theoretical and practical significance to consider weight-unbalanced directed networks. In the discrete-time case, some push-sum and push-pull based strategies are proposed in \cite{Nedic2014distributed,Xi2016distributed,Xi2017dextra,Xi2018linear} to tackle general directed graphs by exploiting a row or column stochastic matrix. Inspired by the graph balancing technique in \cite{ren2005consensus,mei2015distributed}, a distributed continuous-time control strategy is proposed in \cite{Li2017distributed} by utilizing the left eigenvector corresponding to the zero eigenvalue of the Laplacian matrix to tackle weight-unbalanced directed networks. However, the technique cannot be adopted when the left eigenvector is not known \textit{a priori}. Furthermore, to remove the explicit dependency on the left eigenvector corresponding to the zero eigenvalue of the Laplacian matrix, the authors in \cite{Zhu2018continuous} propose a novel distributed algorithm with its gradient term being divided by an auxiliary variable. One limitation in the above works is that only single integrator agent dynamics are considered.

This paper considers the OOC problem of multi-agent systems with second-order nonlinear agent dynamics in the presence of uncertainties and disturbances over weight-unbalanced directed networks. Generally, main challenges in solving the problem arise from uncertain nonlinear dynamics/disturbances and unbalanced directed graphs. To address these challenges, we first convert the OOC problem into a reference-tracking problem by designing a distributed coordinator and then stabilize the obtained augmented system with a state feedback controller. The contributions of this paper in comparison to those existing relevant works are summarized as follows.

1) In contrary to undirected or balanced directed graphs considered in \cite{Wang2010control,Gharesifard2013distributed,Kia2015distributed,tang2020optimal,li2020distributed,li2021distributed}, this work concentrates on the more general and also more challenging weight-unbalanced directed networks. It is shown that the proposed controller is capable of tackling the imbalance arising from general directed networks, and thus is expected to be more widely applicable. 

2) Compared to integrator-type or linear agent dynamics considered in \cite{Gharesifard2013distributed,Wang2010control,Kia2015distributed,Zhu2018continuous,Xie2020suboptimal,Li2017distributed}, nonlinear agent dynamics with uncertainties and external disturbances are studied in this work. The adaptive control technique is exploited in controller design for practical cases where no prior knowledge of the uncertainties is available. In addition, an internal model is designed for each agent to deal with the external disturbance, which is not considered in \cite{tang2020optimal,li2020distributed,li2021distributed}. 
	 
3) Different from the existing works \cite{li2021distributed,li2018consensus,li2020distributed} where nonlinear functions are required to be globally Lipschitz, this work does not suffer from such a restriction. It is thus expected that the distributed controller developed in this paper can be applied more widely in practice.\\[-2mm]
 
The rest of this paper is organized as follows. Some necessary preliminaries are firstly reviewed in Section \ref{section preliminaries}. The problem formulation and main results of this paper are then given in Section \ref{section problem formulation} and Section \ref{section main results}, respectively. Two simulation examples are provided in Section \ref{section simulation results} to illustrate the effectiveness of the proposed controller, and the conclusion and future challenges are finally given in Section \ref{section conclusion}.

\textit{Notations}: $\mathbb{R}$, $\mathbb{R}^{n}$ and $ \mathbb{R}^{N \times N} $ refer to the sets of real numbers, $ n $-dimensional real vectors and $ N $-dimensional real square matrices, respectively. $\mathbf{0}_{N}$ and $\mathbf{1}_{N}$ are used to describe the $ N $-dimensional column vector with all entries equal to $ 0 $ and $ 1 $, respectively. $I_{n}$ represents the identity matrix of dimension $n\times n$. $\|\cdot\|$ denotes the Euclidean norm of vectors or induced 2-norm of matrices. $ A_{i} $ and $ A_{i}^{j} $ represent the $ i $-th row elements and the $ (i,j) $ entry of matrix $ A $, respectively. For matrices $ A $ and $ B $, their Kronecker product is denoted as $ A\otimes B $. $x^{\mathrm{T}}$ and $A^{\mathrm{T}}$ refer to the transpose of vector $ x $ and matrix $ A $, respectively. $ \operatorname{col}\left(x_{1}, x_{2}, \ldots, x_{n}\right) $ represents a column vector with $x_{1}, x_{2}, \ldots, x_{n}$ being its elements. $\operatorname{diag}\left(x_{1}, x_{2}, \ldots, x_{n}\right)$ represents a diagonal matrix with $x_{1}, x_{2}, \ldots, x_{n}$ being its diagonal elements. For a differentiable function $f: \mathbb{R}^{n} \rightarrow \mathbb{R}$, $\nabla f$ denotes its gradient. A continuous function $\alpha:[0, a) \rightarrow[0, \infty)$ is said to belong to class $\mathcal{K}$ if it is strictly increasing and $\alpha(0)=0 .$ It is said to belong to class $\mathcal{K}_{\infty}$ if it belongs to class $\mathcal{K}$ and $\lim_{r\to\infty} \alpha(r) = \infty$.

\section{Preliminaries}\label{section preliminaries}
In this section, we present some preliminaries on graph theory, convex analysis, and perturbed system theory.

\subsection{Graph Theory}
A graph is used to represent the information flow between agents. A weighted directed graph (in short, a digraph) of order $ N $ can be described by a triplet $\mathcal{G}=(\mathcal{V}, \mathcal{E}, \mathcal{A})$, which consists of a set $\mathcal{V}=\{1, \ldots, N\}$ of nodes, a collection $\mathcal{E} \subseteq \mathcal{V} \times \mathcal{V}$ of ordered pairs of nodes, called edges, and a weighted adjacency matrix $\mathcal{A}$. 
For $ i,j\in\mathcal{V} $, the ordered pair $(j, i) \in \mathcal{E}$ denotes an edge from $ j $ to $ i $, which means that the $i$-th agent can receive information from the $j$-th agent, but not vice versa. In this case, node $ j $ is called an in-neighbor of node $ i $, and node $ i $ is called an out-neighbor of node $ j $. The in-degree $ d_{\textrm{in}}(i) $ and out-degree $ d_{\textrm{out}}(i) $ of agent $ i $ are the numbers of its in-neighbors and out-neighbors, respectively.
In a digraph, a directed path is an ordered sequence of nodes in which any pair of consecutive nodes is a directed edge. A self-loop is an edge from a node to itself. Consistent with a common convention, it is assumed that there is no self-loop in a digraph. 
A digraph is said to be strongly connected if for any node, there exists a directed path from any other node to itself. 
The weighted adjacency matrix is denoted as $\mathcal{A}=\left[a_{i j}\right] \in \mathbb{R}^{N \times N}$, where $a_{i j}>0$ if $(j, i) \in \mathcal{E},$ otherwise $a_{i j}=0 .$ Besides, $a_{i i}=0$ for all $ i $ since there is no self-loop. Moreover, the Laplacian matrix $\mathcal{L}=\left[l_{i j}\right] \in \mathbb{R}^{N \times N}$ associated with the digraph $\mathcal{G}$ is defined as $l_{i i}=\sum_{j=1}^{N} a_{i j}$ and $l_{i j}=-a_{i j}$ for $i \neq j $. A digraph $\mathcal{G}$ is weight balanced if and only if $\mathbf{1}_{N}^{\mathrm{T}} \mathcal{L}=\mathbf{0}_{N}^{\mathrm{T}}$. One may refer to \cite{Bullo2019lectures} for more details on graph theory.

\begin{lemma}\cite{Bullo2019lectures,mei2015distributed} \label{graph theory lemma}
	Let $ \mathcal{L} $ be the Laplacian matrix associated with a strongly connected directed graph $ \mathcal{G} $. Then the following statements hold.
	\begin{itemize}
		\item[\romannumeral1)] \label{graph theory lemma_1} There exists a positive left eigenvector $ \varrho=\left(\varrho_{1},\varrho_{2},\ldots,\varrho_{N} \right)^{\mathrm{T}}  $ associated with the zero eigenvalue of the Laplacian matrix such that $ \varrho^{\mathrm{T}}\mathcal{L}=\mathbf{0}_{N}^{\mathrm{T}} $ and $ \sum_{i=1}^{N}\varrho_{i}=1 $.
		\item[\romannumeral2)] \label{graph theory lemma_2} Let $ R=\mathrm{diag}\left(\varrho_{1},\varrho_{2},\ldots,\varrho_{N} \right) $ and $\bar{\mathcal{L}}=\big( R\mathcal{L}+\mathcal{L}^{\mathrm{T}}R\big)/2 $. Then $ \bar{\mathcal{L}} $ is positive semidefinite, and its eigenvalues can be ordered as $ 0=\lambda_{1}<\lambda_{2}\leq \lambda_{3}\leq\ldots \leq\lambda_{N} $. 
		\item[\romannumeral3)] \label{graph theory lemma_3} $ \exp (-\mathcal{L} t) $ is a nonnegative matrix with positive diagonal entries for all $ t>0 $, and $ \lim _{t \rightarrow \infty} \exp (-\mathcal{L} t)=\mathbf{1}_{N} \varrho^{\mathrm{T}} $.
	\end{itemize}
\end{lemma}

\subsection{Convex Analysis}
In this subsection, the definitions of Lipschitz continuity and strong convexity are recalled, please see \cite{Bertsekas2009convex} for more details.

A differentiable function $c: \mathbb{R}^{n} \rightarrow \mathbb{R}$ is said to be $\varpi $-strongly convex on $\mathbb{R}^{n}$ if there exists a constant $ \varpi>0 $ such that $(x-y)^{\operatorname{T}}(\nabla c(x)-\nabla c(y)) \geq \varpi\|x-y\|^{2}$ for all $x, y \in \mathbb{R}^{n} $. A function $g: \mathbb{R}^{n} \rightarrow \mathbb{R}^{n}$ is said to be globally Lipschitz on $\mathbb{R}^{n}$ if there exists a constant $ \iota>0 $ such that $\|g(x)-g(y)\| \leq \iota\|x-y\|$ for all $x, y \in \mathbb{R}^{n}$.

\subsection{Perturbed System Theory}
Last but not least, a theory of perturbed system which facilitates subsequent analysis is recalled in this subsection.

\begin{lemma} \label{lemma 4}
	Consider the following perturbed system
	\begin{equation} \label{lemma_perturbed system}
		\dot{x}=f(t,x)+g(t,x).
	\end{equation}
	Let $x=0$ be an exponentially stable equilibrium point of the nominal system $ \dot{x}=f(t,x) $, where $ f $ is continuously differentiable and its Jacobian matrix $ [\partial f/ \partial x] $ is bounded on $ \mathbb{R}^{n} $. Suppose the perturbation term $g(t, x)$ satisfies $ g(t,0)=0 $ and $ \|g(t, x)\| \leq \gamma(t)\|x\| $, where $\lim_{t\to\infty}\gamma(t)= 0$. Then, the origin is an exponentially stable equilibrium point of the perturbed system (\ref{lemma_perturbed system}). 
	
	\begin{proof}
		It can be proved by a simple combination of Corollary 9.1 and Lemma 9.5 in \cite{khalil2002nonlinear}.
	\end{proof}
\end{lemma}

\section{Problem Formulation} \label{section problem formulation}

Consider a heterogeneous multi-agent system composed of $ N $ agents over a weight-unbalanced directed network. The dynamics of the agents are described by the following second-order uncertain nonlinear systems,
\begin{align} \label{dynamics}
	\dot{x}_{i 1} & =x_{i 2}, \notag\\
	\dot{x}_{i 2} & =f_{i}\left(x_{i 1}, x_{i 2}, v, w\right)+b_{i}(w) u_{i},\\
	y_{i} & =x_{i 1}, \quad i=1,2, \ldots, N, \notag
\end{align}
where $ x_{i}=\operatorname{col}(x_{i1},x_{i2}) \in\mathbb{R}^{2}$ is the state of the $ i$-th agent, $u_{i} \in \mathbb{R}$ and $ y_{i} \in \mathbb{R}$ are its control input and measurement output, respectively. $ w \in\mathbb{W} $ represents an uncertain parameter vector, with $ \mathbb{W} \subset \mathbb{R}^{n_{w}}$ being an unknown compact set. $v \in \mathbb{R}^{n_{v}}$ is an exogenous signal representing the disturbance generated by the following uncertain linear exosystem,
\begin{align} \label{exosystem}
\dot{v}=S(\sigma) v,
\end{align}
where $ \sigma\in\mathbb{R}^{n_{\sigma}} $ represents an uncertain constant vector  belonging to an unknown compact set $ \mathbb{S} $. For $ i=1,2, \ldots, N $, $f_{i}$ and $ b_{i} $ are assumed to be sufficiently smooth functions satisfying $f_{i}(0,0,0, w)=0$ and $ b_{i}(w)>0 $ for all $ w\in\mathbb{W} $.

\begin{remark}
	Compared with existing works, the nonlinear dynamics considered in this work are more general in at least two aspects. On one hand, unlike \cite{li2021distributed,li2018consensus} where the nonlinear functions $ f_{i} $'s are required to satisfy the globally Lipschitz condition, this work does not suffer from such a restriction. On the other hand, the nonlinear dynamics considered in this work contain external disturbances generated by an exosystem (\ref{exosystem}), which cannot be handled in \cite{tang2020optimal,li2018consensus,li2021distributed}. It is also worth pointing out that the exosystem (\ref{exosystem}) can produce a large class of external signals, such as sinusoidal, step and ramp type signals \cite{huang2004nonlinear}.
\end{remark}

In addition, it is assumed that each agent $ i $ possesses a local cost function $ c_{i}(s): \mathbb{R}\to \mathbb{R} $ with its local decision variable $ s\in \mathbb{R} $. It should be emphasized that the local cost function $ c_{i}(\cdot) $ is only avaliable to agent $ i $. Define the global cost function and its corresponding optimal solution as $ c(s)=\sum_{i=1}^{N}c_{i}(s) $ and $ s^{\star}\in \mathbb{R} $, respectively. 
To seek the global minimizer in a distributed manner, the controller design for each agent is only allowed to make use of information from its in-neighbors and itself. More specifically, the controller is expected to take the following form,
\begin{equation} \label{controller form_origin system}
	\begin{aligned}
	u_{i} = & \kappa_{i1}\big(\nabla f_{i}, x_{i}, \upsilon_{j}, j\in \bar{\mathcal{N}}_{i}\big), \\
	\dot{\upsilon}_{i} = & \kappa_{i2}\big(\nabla f_{i}, x_{i}, \upsilon_{j}, j\in \bar{\mathcal{N}}_{i}\big),
	\end{aligned}
\end{equation}
where $ \kappa_{i1} $ and $ \kappa_{i2} $ are sufficiently smooth functions vanishing at the origin, $ \bar{\mathcal{N}}_{i} = \mathcal{N}_{i} \cup \{i\} $ is a set containing the in-neighbors and itself of agent $ i $, $ \upsilon_{i}\in \mathbb{R}^{n_{i\upsilon}} $ is a state of the dynamic controller with its dimention $ n_{i\upsilon} $ to be specificed later.\\[-1mm]

The objective of this work is to develop a distributed state feedback controller such that the following defined distributed optimal output consensus problem for second-order uncertain nonlinear multi-agent systems over weight-unbalanced directed networks can be solved.
\begin{problem} \label{problem}
	Consider the multi-agent system (\ref{dynamics}) and the exosystem (\ref{exosystem}) under the directed graph $ \mathcal{G} $ with local cost functions $ c_{i}(\cdot) $'s, and nonempty compact sets $ \mathbb{W}\subseteq \mathbb{R}^{n_{w}} $ and $ \mathbb{S}\subseteq \mathbb{R}^{n_{\sigma}} $, which are not necessarily known. Design a distributed dynamic state feedback controller of the form (\ref{controller form_origin system}) such that, for any constant vector $ \operatorname{col}(w,\sigma)\in \mathbb{W}\times \mathbb{S} $, the trajectories of the closed-loop system composed of the dynamics (\ref{dynamics}) and the distributed controller (\ref{controller form_origin system}) starting form any initial state $ x_{i}(0)=x_{0} $ and $ v(0)=v_{0} $ exists and is bounded for all $ t\geq 0 $, moreover, the outputs $ y_{i}, i=1,2,\ldots,N $ converge to the optimal value $ s^{\star} $ which minimizes the global cost function $ c(s) = \sum_{i=1}^{n}c_{i}(s) $ as time goes to infinity.
\end{problem}

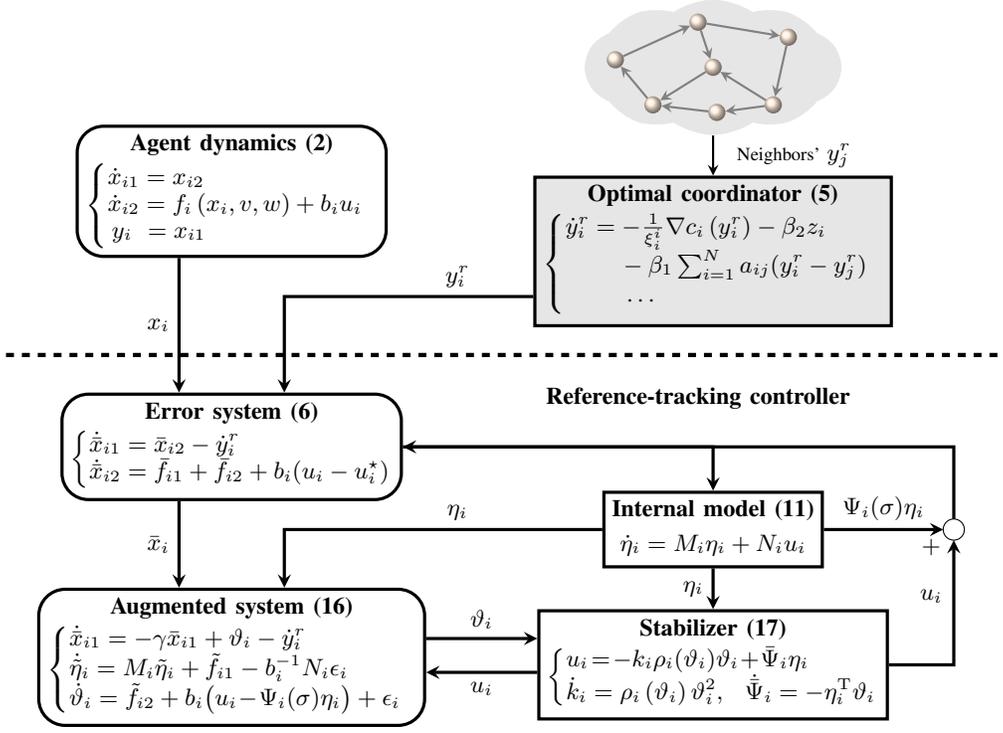
\begin{figure*}[!t]
	\centering
	\small
	\begin{tikzpicture}[scale=1, auto, every edge/.append style = { -, thick, draw=gray}]
		
		\tikzset{VertexStyle/.style = {shape = circle, ball color = orange!15, text = black, inner sep = 2.3pt, outer sep = 0pt, minimum size = 5 pt}}
		
		\tikzstyle{arrow} = [thick,->,>=stealth]
		
		\tikzstyle{block} = [draw, rectangle, very thick, minimum height=2em, minimum width=2em, align = center, font = \bfseries]
		\tikzstyle{blockdash} = [draw, dash, rectangle, ultra thick, minimum height=2em, minimum width=2em, align = center, font = \bfseries]
		\tikzstyle{corners} = [draw, rectangle, rounded corners=10pt, very thick, minimum height=2em, minimum width=2em, align = center, font = \bfseries]
		
		\tikzstyle{sum} = [draw,  circle, minimum size = 8pt, inner sep = 0pt, outer sep = 0pt]
		\node (Dynamics) [corners,align=center] {\textbf{Agent dynamics (\ref{dynamics})}\\[1mm] $
			\left\{\begin{array}{l}
				\!\!\!\dot{x}_{i 1} =x_{i 2} \\
				\!\!\!\dot{x}_{i 2} =f_{i}\left(x_{i}, v, w\right)+b_{i} u_{i} \\
				\!\!y_{i} ~=x_{i 1}
			\end{array}\right.
			$};
		
		\node (Error) [corners,align=center, below of=Dynamics, node distance=3.4cm] {\textbf{Error system (\ref{error_system})}\\[1mm] $
			\left\{\begin{array}{l}
				\!\!\!\dot{\bar{x}}_{i 1}=\bar{x}_{i 2}-\dot{y}_{i}^{r} \\
				\!\!\!\dot{\bar{x}}_{i 2}=\bar{f}_{i1} +\bar{f}_{i 2} +b_{i} (u_{i}-u_{i}^{\star})
			\end{array}\right.\!\!\!
			\vspace{0.05cm}
			$};
		
		\node (Augmented) [corners,align=center, below of=Error, node distance=2.8cm] {\textbf{Augmented system (\ref{augmented system 2})}\\[1mm] $
			\left\{\begin{array}{l}
				\!\!\!\dot{\bar{x}}_{i 1}=-\gamma \bar{x}_{i 1}+\vartheta_{i}-\dot{y}_{i}^{r} \\
				\!\!\!\dot{\tilde{\eta}}_{i}=M_{i} \tilde{\eta}_{i}+\tilde{f}_{i 1}-b_{i}^{-1} N_{i}\epsilon_{i} \\
				\!\!\!\dot{\vartheta}_{i}=\tilde{f}_{i 2} +b_{i} \big(u_{i}\!-\! \Psi_{i}(\sigma)\eta_{i} \big) +\epsilon_{i}
			\end{array}\right.
			\vspace{0.05cm}
			$};
		
		\node (Coordinator) [block,align=center, right of=Dynamics, node distance=6.4cm, yshift=-0.8cm, fill=black!10] {\textbf{Optimal coordinator (\ref{algorithm_r})}\\[1mm] $
			\left\{\begin{array}{l}
				\!\!\!\dot{y}_{i}^{r} =-\frac{1}{\xi_{i}^{i}} \nabla c_{i}\left(y_{i}^{r}\right)-\beta_{2} z_{i} \\ 
				~~~~~\!-\beta_{1} \sum_{i=1}^{N} a_{i j}(y_{i}^{r}-y_{j}^{r}) \\[1mm]
				~~~~~\cdots
			\end{array}\right.
			\vspace{0.05cm}
			$};
		
		\node (IM) [block,align=center, below of=Coordinator, node distance=3.7cm] {\textbf{Internal model (\ref{internal_model})}\\[1mm] $
			\dot{\eta}_{i}=M_{i} \eta_{i}+N_{i} u_{i}
			\vspace{0.05cm}
			$};
		
		\node (Stabilizer) [block,align=center, below of=IM, node distance=1.8cm] {\textbf{Stabilizer (\ref{controller_augmented_1})}\\[1mm] $
			\left\{\begin{array}{l}
				\!\!\!u_{i} \!=\!-k_{i} \rho_{i}\!\left(\vartheta_{i}\right)\! \vartheta_{i} \!+\! \bar{\Psi}_{i} \eta_{i} \\
				\!\!\!\dot{k}_{i} =\rho_{i}\left(\vartheta_{i}\right) \vartheta_{i}^{2},~~
				\dot{\bar{\Psi}}_{i} =-\eta_{i}^{\operatorname{T}}\vartheta_{i}
			\end{array}\right.\!\!\!
			\vspace{0.05cm}
			$};
		
		\node [sum, right of=IM, node distance=3.2cm] (sum) {};

		\node(cloud)[cloud, cloud puffs=10, cloud puff arc= 70,
		minimum width=3.4cm, minimum height=1.8cm, aspect=1, above of= Coordinator, xshift=-0cm, node distance=2.45cm, fill=black!10]  {};

		\path [draw, arrow, thick]  (cloud.south) -- node [right, align=center,xshift=0.2cm,] {{\scriptsize Neighbors' }${y}_{j}^{r}$} (Coordinator.north -| cloud.south);
		
		\foreach \place/\x in {{(-2.1,0.3)/1}, {(-1.6,-0.3)/2},{(-1,0.8)/3},
			{(-0.75,-0.4)/4}, {(-0.8,0.2)/5}, {(0.2,0.6)/6}, {(0,-0.3)/7}}
		\node[VertexStyle, above of= Coordinator, xshift=7.2cm, node distance=1.45cm] (a\x) at \place {};
		
		\draw[arrow,black!50] (a2) -- (a1);
		\draw[arrow,black!50] (a1) -- (a3);
		\draw[arrow,black!50] (a7) -- (a5);
		\draw[arrow,black!50] (a3) -- (a6);
		\draw[arrow,black!50] (a4) -- (a2);
		\draw[arrow,black!50] (a3) -- (a5);
		\draw[arrow,black!50] (a7) -- (a4);
		\draw[arrow,black!50] (a5) -- (a2);
		\draw[arrow,black!50] (a6) -- (a7);
		
		\path (Dynamics.south)+(-0.7,0) coordinate(ds);
		\path (Error.north)+(-0.7,0) coordinate(en1);
		\draw [arrow, -stealth, very thick](ds) -- node [left] {$x_{i}$} (en1);
		
		\path (Error.north)+(0.7,0) coordinate(en2);
		\path (Coordinator.west)+(0,-0.6) coordinate(cw);
		\draw [arrow, very thick](cw) -| node [above, xshift=2.3cm] {$y_{i}^{r}$} (en2);
		
		\path (Error.south)+(-0.7,0) coordinate(es);
		\path (Augmented.north)+(-0.7,0) coordinate(an1);
		\draw [arrow, very thick](es) -- node [left] {$\bar{x}_{i}$} (an1);
		
		\path (Augmented.north)+(0.7,0) coordinate(an2);
		\draw [arrow, very thick](IM.west) -| node [above, xshift=2.3cm] {$\eta_{i}$} (an2);
		
		\path (Augmented.east)+(0,0.25) coordinate(ae);
		\path [draw, arrow, very thick]  (ae) -- node [above, align=center] {$\vartheta_{i}$} (Stabilizer.west |- ae);
		
		\path (Stabilizer.west)+(0,-0.1) coordinate(sw);
		\path [draw, arrow, very thick]  (sw) -- node [below, align=center] {$u_{i}$} (Augmented.east |- sw);
		
		\draw [arrow, very thick](IM.south) -- node [left] {$\eta_{i}$} (Stabilizer.north);
		
		\draw [arrow, very thick](IM.east) -- node[below,pos=0.9] {$+$} node [above] {$\Psi_{i}(\sigma)\eta_{i}$} (sum);
		
		\draw [arrow, very thick](Stabilizer.east) -| node [above, xshift=-0.3cm, yshift=0.7cm] {$u_{i}$} (sum);
		
		\draw [arrow, very thick](sum) |- node[above, xshift=-0.2cm, yshift=0cm] {$$} (Error.east);
		
		\draw [arrow, very thick](Error.east) -| node[] {} (IM.north);
		
		\path (Error.north)+(-3,0.5) coordinate(en3);
		\path (Error.north)+(10.4,0.5) coordinate(en4);
		\draw [ultra thick, dashed](en3) -- node [below,xshift=2.5cm,yshift=-0.3cm] {\textbf{Reference-tracking controller}} (en4);
		
	\end{tikzpicture}
	\caption{Design diagram for distributed optimal output consensus over weight-unbalanced directed networks.}
	\label{Fig_Two_layer}
\end{figure*}

To solve Problem \ref{problem}, the following standard assumptions are needed.

\begin{assumption} \label{assumption_cost functions}
	The local cost function $ c_{i} $ is continuously differentiable and $ \varpi_{i} $-strongly convex, and $\nabla c_{i}$ is globally Lipschitz on $ \mathbb{R} $ with constant $ \iota_{i} $.
	\label{cost function assumption}
\end{assumption}

\begin{remark} \label{remark_cost function_2}
	The strong convexities of local cost functions in Assumption \ref{assumption_cost functions} guarantee the existence and uniqueness of the optimal solution $ s^{\star}\in \mathbb{R} $. Assumption \ref{assumption_cost functions} is standard for solving the distributed optimization problem, and thus commonly used in many existing works, for example, \cite{Kia2015distributed,Zhu2018continuous,tang2020optimal}.
\end{remark}

\begin{assumption} \label{graph assumption}
	The directed graph $ \mathcal{G} $ is strongly connected.
\end{assumption}

\begin{assumption} \label{assumption_exosysem}
	The exosystem is neutrally stable, i.e., all the eigenvalues of $ S(\sigma) $ are semi-simple with zero real parts for all $ \sigma\in\mathbb{S} $.
\end{assumption}

\begin{remark} \label{remark_d}
	Under Assumption \ref{assumption_exosysem}, it can be shown that given any initial condition $ v(0)\in \mathbb{R}^{n_{v}} $, the state $ v(t) $ of the exosystem (\ref{exosystem}) is bounded for all $ t\geq 0 $. Therefore, there exists a compact set $ \mathbb{D}\subseteq \mathbb{R}^{n_{v}}\times\mathbb{W} \times \mathbb{S} $ to which $ d=\operatorname{col}(v,w,\sigma) $ belongs for all $ t\geq 0 $. It is worth emphasizing that we only rely on the compactness of the set $ \mathbb{D} $ in the follow-up analysis, and do not need a prior knowledge of its exact bound.
\end{remark}

\section{Main Results}\label{section main results}

In this section, a distributed optimal output consensus strategy based on a two-layer structure is developed to solve the problem by addressing the difficulties resulting from weight-unbalanced directed graphs and uncertain nonlinear dynamics. The architecture of the two-layer strategy is depicted in Fig. \ref{Fig_Two_layer}. The upper layer is a distributed optimal coordinator for each agent, which cooperates with others to generate a local reference signal $ y_{i}^{r} $ that eventually converges to the optimal solution, while the lower layer is a reference-tracking controller for each agent to track its private reference signal $ y_{i}^{r} $. 

In what follows, the distributed optimal coordinator for each agent will be first developed. An associated error system is then defined such that the Problem \ref{problem} can be converted to a reference-tracking problem. Then by adopting an internal model to handle the external disturbance in agent dynamics, an adaptive stabilizer will be proposed to deal with the augmented system composed of the obtained error system and internal model.

\subsection{Distributed Optimal Coordinator Design}
In this subsection, we design the distributed optimal coordinator to generate the optimal solution for the concerned multi-agent system over the unbalanced directed network. Then by embedding the optimal coordinator in the feedback loop, the optimal output consensus problem under consideration is converted to a reference-tracking problem, which will be addressed in the next subsection. Specifically, inspired by \cite{Zhu2018continuous}, the optimal coordinator is designed for each agent $ i $ as follows,
\begin{subequations} \label{algorithm_r}
	\begin{align}
	\dot{y}_{i}^{r} & =-\frac{1}{\xi_{i}^{i}} \nabla c_{i}\left(y_{i}^{r}\right)-\beta_{1} \sum_{i=1}^{N} a_{i j}\left(y_{i}^{r}-y_{j}^{r}\right)-\beta_{2} z_{i},  \label{algorithm_r_a}\\
	\dot{z}_{i} & =\beta_{1} \sum_{i=1}^{N} a_{i j}\left(y_{i}^{r}-y_{j}^{r}\right), \quad z_{i}(0)=0,\\
	\dot{\xi}_{i} &= -\sum_{j=1}^{N}a_{ij}(\xi_{i}-\xi_{j}),
	\end{align}
\end{subequations}
where $ y_{i}^{r}\in \mathbb{R} $ represents the generated reference signal for agent $ i $, $ z_{i}\in \mathbb{R} $  and $ \xi_{i}\in\mathbb{R}^{N} $ are auxiliary variables, with $ \xi_{i}^{k} $ being its $ k $-th component and initial value $ \xi_{i}(0) $ satisfying
$ \xi_{i}^{i}(0) = 1 $, otherwise $ \xi_{i}^{k}(0) = 0 $ for all $ k \neq i $. $ \beta_{1} $ and $ \beta_{2} $ are two positive constants to be determined later. The optimal coordinator for each agent only requires information of its neighbors and itself, and thus is distributed.

Define $\xi=\operatorname{col}\left(\xi_{1},\xi_{2}, \ldots, \xi_{N}\right)$. Then the dynamics of $\xi$ can be written as $\dot{\xi}=-\left(\mathcal{L} \otimes I_{N}\right) \xi $. Recall that $\xi_{i}^{i}(0)=1$, and $\xi_{i}^{k}(0)=0$ for all $k \neq i$. It then follows from \romannumeral3) of Lemma \ref{graph theory lemma} that
$$
\begin{aligned}
	\xi_{i}^{i}(t) &=\left[\exp \left(-\left(\mathcal{L} \otimes I_{N}\right) t\right)\right]_{(i-1) N+i} \cdot \xi(0) \\
	&=\left[\exp \left(-\left(\mathcal{L} \otimes I_{N}\right) t\right)\right]_{(i-1) N+i}^{(i-1) N+i} \cdot \xi_{i}^{i}(0)>0,
\end{aligned}
$$
for all $t \geq 0$. Therefore, the algorithm (\ref{algorithm_r_a}) is well defined. Moreover, it can be deduced that $ \lim_{t\to\infty}\xi(t)=\lim_{t\to\infty}\exp\big(-(\mathcal{L}\otimes I_{N})t \big)\xi(0)=  \left( \mathbf{1}_{N}\varrho^{\mathrm{T}}\otimes I_{N}\right)\xi(0)=\mathbf{1}_{N}\otimes \varrho  $, which implies that $ \lim_{t\to\infty} \xi_{i}^{i}(t) = \varrho_{i} $ exponentially.

\begin{remark}
	The design of the optimal coordinator (\ref{algorithm_r}) takes the modified Lagrange structure as proposed in \cite{Kia2015distributed}, with the gradient term $ \nabla c_{i} $ being divided by $ \xi_{i}^{i} $ to address the imbalance resulting from general directed networks. A similar approach in \cite{Li2017distributed} is to divide the gradient term by the left eigenvector $ \varrho $ corresponding to the zero eigenvalue of the Laplacian matrix. However, it is worth noting that the left eigenvector $ \varrho $ is of global information, which may not be available in advance. Hence, the utilization of $ \xi_{i}^{i} $ provides an alternative so that the explicit dependence of the left eigenvector $ \varrho $ can be avoided.
\end{remark}

Now, we show that the generated signals $ y_{i}^{r} $, $ i=1,2,\ldots,N $ eventually tend to the optimal value $ s^{\star} $, summarized in the following result.

\begin{theorem} \label{proposition1} 
	Consider the distributed optimal coordinators (\ref{algorithm_r}) under Assumptions \ref{assumption_cost functions} and \ref{graph assumption}. Given $ z_{i}(0)=0 $ and any initial conditions $ y_{i}^{r}(0) $ for $ i=1,2,\ldots,N $, there exist sufficiently large positive constants  $ \beta_{1}$ and $ \beta_{2} $ such that the generated reference signals $ y_{i}^{r} $'s are bounded for all $ t\geq 0 $, and exponentially converge to the optimal solution $ s^{\star} $ of the global cost function.
\end{theorem}

\begin{proof}
	The proof is given in Appendix \ref{Appendix_1}.
\end{proof}

Next, by taking $ y_{i}^{r} $ as a reference to be tracked by agent $ i $, the optimal output consensus problem can be converted to a reference-tracking problem. To this end, an error system is defined by embedding the optimal coordinator (\ref{algorithm_r}) into the second-order agent dynamics (\ref{dynamics}). More specifically,  let $x_{i}^{\star}=\operatorname{col}(y_{i}^{r}, 0)$ and introduce the transformation $ \bar{x}_i=x_{i}-x_{i}^{\star} $, where $ \bar{x}_{i}= \operatorname{col}(\bar{x}_{i 1}, \bar{x}_{i 2}) $. Then, the dynamics of $ \bar{x}_{i 1} $ and $ \bar{x}_{i 2} $ can be described as follows,
\begin{equation} \label{error_system}
\begin{aligned}
	\dot{\bar{x}}_{i 1}=&\bar{x}_{i 2}-\dot{y}_{i}^{r}, \\
	\dot{\bar{x}}_{i 2}=&\bar{f}_{i1}\left(\bar{x}_{i 1}, \bar{x}_{i 2}, y_{i}^{r}, d\right)+\bar{f}_{i 2}\left(y_{i}^{r}, s^{\star}, d\right)\\
	&+b_{i}(w)\left(u_{i}-u_{i}^{\star}\left(s^{\star}, d\right)\right),
\end{aligned}
\end{equation}
where $ d=\operatorname{col}(v,w,\sigma) $, $\bar{f}_{i1}\left(\bar{x}_{i 1}, \bar{x}_{i 2}, y_{i}^{r}, d\right)=f_{i}\left(x_{i 1}, x_{i 2}, v, w\right)-f_{i}\left(y_{i}^{r}, 0, v, w\right)$, $\bar{f}_{i 2}\left(y_{i}^{r}, s^{\star}, d\right)=f_{i}\left(y_{i}^{r}, 0, v, w\right)-f_{i}\left(s^{\star}, 0, v, w\right)$, $u_{i}^{\star}\left(s^{\star}, d\right)=-f_{i}\left(s^{\star}, 0, v, w\right) / b_{i}(w)$. Recalling that $ b_{i}(w)>0 $ for all $ w\in\mathbb{W} $, $ u_{i}^{\star}\left(s^{\star}, d\right) $ is well-defined. Moreover, it follows from the smoothness of $ f_{i} $ that $ \bar{f}_{i1} $ and $ \bar{f}_{i2} $ are sufficiently smooth functions satisfyinig $ \bar{f}_{i1}\left(0,0, y_{i}^{r}, d\right)=0$ for all $ y_{i}^{r}\in\mathbb{R} $ and $ d\in \mathbb{D} $. 

Consider a dynamic controller of the following form for the error system (\ref{error_system}),
\begin{equation} \label{controller form_error system}
\begin{aligned}
u_{i} = & \bar{\kappa}_{i1}\big(\bar{x}_{i}, \bar{\upsilon}_{j}, j\in \bar{\mathcal{N}}_{i}\big), \\
\dot{\bar{\upsilon}}_{i} = & \bar{\kappa}_{i2}\big(\bar{x}_{i}, \bar{\upsilon}_{j}, j\in \bar{\mathcal{N}}_{i}\big),
\end{aligned}
\end{equation}
where $ \bar{\kappa}_{i1} $ and $ \bar{\kappa}_{i2} $ are sufficiently smooth functions vanishing at the origin, $ \bar{\mathcal{N}}_{i} = \mathcal{N}_{i} \cup \{i\} $ is defined to be the same as that in (\ref{controller form_origin system}), and $ \bar{\upsilon}_{i}\in \mathbb{R}^{n_{i\bar{\upsilon}}} $ is the state of the dynamic controller with its dimention $ n_{i\bar{\upsilon}} $ to be specificed later. Now we are ready to present the following reference-tracking problem. 

\begin{problem} \label{problem_2}
	Given the error system (\ref{error_system}), the optimal coordinator (\ref{algorithm_r}), and nonempty compact sets $ \mathbb{W}\subseteq \mathbb{R}^{n_{w}} $ and $ \mathbb{S}\subseteq \mathbb{R}^{n_{\sigma}} $, which are not necessarily known, design a distributed dynamic feedback controller of the form (\ref{controller form_error system}) such that, for any constant vector $ \operatorname{col}(w,\sigma)\in \mathbb{W}\times \mathbb{S} $, the trajectories of the closed-loop system composed of the error system (\ref{error_system}), the optimal coordinator (\ref{algorithm_r}) and the dynamic controller (\ref{controller form_error system}) is bounded for all $ t\geq 0 $, and the error states $ \bar{x}_{i} $'s tend to the origin as time goes to infinity.
\end{problem}

The following lemma shows that the optimal output consensus Problem \ref{problem} is solved as long as the reference-tracking Problem \ref{problem_2} is solved.
\begin{lemma} \label{lemma_stabilization}
	Suppose that Assumptions \ref{assumption_cost functions}-\ref{assumption_exosysem} hold and there exists a smooth dynamic controller of the form (\ref{controller form_error system}) that solves the Problem \ref{problem_2} for the error system (\ref{error_system}). Then, Problem \ref{problem} can be solved by a distributed dynamic controller composed of (\ref{algorithm_r}) and (\ref{controller form_error system}).
\end{lemma}

\begin{proof}
	By the triangle inequality, one has $ |y_{i}-s^{\star}|\leq |y_{i}-y_{i}^{r}| + |y_{i}^{r}-s^{\star}| $. It follows from Theorem \ref{proposition1} that $ y_{i}^{r} $ is bounded for all $ t\geq 0 $ and $ \lim_{t\to\infty} y_{i}^{r} = s^{\star} $. By the definition of $ \bar{x}_{i} $, one obtains that $ \lim_{t\to\infty} y_{i} = y_{i}^{r} $ as long as Problem \ref{problem_2} is solved. Therefore, one can conclude that $ \lim_{t\to\infty} y_{i} = s^{\star} $, and Problem \ref{problem} is thus solved.
\end{proof}

\subsection{Reference-Tracking Controller Design}

In this subsection, a novel dynamic state feedback controller is developed to solve the reference-tracking Problem \ref{problem_2}. The design of the reference-tracking controller can be accomplished by the following four steps.

\textbf{Step 1.} The error system (\ref{error_system}) is transformed into a new system of relative degree one. To this end, define $\vartheta_{i}=\bar{x}_{i 2}+\gamma \bar{x}_{i 1}$, where $ \gamma>1 $ is a constant to be specified later. Then the dynamics of $ \vartheta_{i} $ can be described as follows, 
\begin{align*}
\dot{\vartheta}_{i} =& \hat{f}_{i}\left(\bar{x}_{i 1}, \vartheta_{i}, y_{i}^{r}, d\right)+\bar{f}_{i 2}\left(y_{i}^{r}, s^{\star}, d\right) \\
&+b_{i}(w)\left(u_{i}-u_{i}^{\star}\left(s^{\star}, d\right)\right)-\gamma \dot{y}_{i}^{r},
\end{align*}
where $\hat{f}_{i}(\bar{x}_{i 1}, \vartheta_{i}, y_{i}^{r}, d)=\bar{f}_{i1}(\bar{x}_{i 1}, \bar{x}_{i 2}, y_{i}^{r}, d)+\gamma(-\gamma \bar{x}_{i 1}+\vartheta_{i})$ is smooth and satisfies $\hat{f}_{i}\left(0,0, y_{i}^{r}, d\right)=0$ for all $ y_{i}^{r}\in\mathbb{R} $ and $ d\in \mathbb{D} $. 

It can be seen that to prove that both $ \bar{x}_{i 1} $ and $ \bar{x}_{i 2} $ asymptotically converge to the origin, it is sufficient to show that $ (\bar{x}_{i 1},\vartheta_{i})=(0,0) $ is an asymptotically stable equilibrium point of the following system of relative degree one,
\begin{equation} \label{error_system_1}
\begin{aligned}
\dot{\bar{x}}_{i 1} =&-\gamma \bar{x}_{i 1}+\vartheta_{i}-\dot{y}_{i}^{r}, \\
\dot{\vartheta}_{i} =& \hat{f}_{i}\left(\bar{x}_{i 1}, \vartheta_{i}, y_{i}^{r}, d\right)+\bar{f}_{i 2}\left(y_{i}^{r}, s^{\star}, d\right)\\
&+b_{i}(w)\left(u_{i}-u_{i}^{\star}\left(s^{\star}, d\right)\right)-\gamma \dot{y}_{i}^{r}.
\end{aligned}
\end{equation}

\textbf{Step 2.} An internal model is designed such that the feed-forward variable $ u_{i}^{\star}\left(s^{\star}, d\right) $ can be asymptotically reproduced. Due to the presence of uncertain parameter $ d $, the feed-forward term $ u_{i}^{\star}\left(s^{\star}, d\right) $ in (\ref{error_system_1}) is unavailable for feedback. To tackle this challenge, we need an additional standard assumption \cite{huang2004nonlinear,Wang2015distributed}.

\begin{assumption} \label{assumption_u}
	The functions $u_{i}^{\star}\left(s^{\star}, d\right) = u_{i}^{\star}\left(s^{\star}, v,w,\sigma\right) $, $ i=1,2,\ldots,N $ are polynomials in $v$ with coefficients depending on $ s^{\star} $, $w$ and $\sigma$.
\end{assumption}

Now, we are ready to design an internal model to generate the feed-forward term $ u_{i}^{\star}\left(s^{\star}, d\right) $. More specifically, under Assumptions \ref{assumption_exosysem} and \ref{assumption_u}, there exist integers $s_{i}, i=1,2, \ldots, N$ such that for all $\operatorname{col}(w, \sigma) \in \mathbb{W} \times \mathbb{S}$, one has
\begin{equation*}
\begin{aligned}
	\frac{\mathrm{d}^{s_{i}} u_{i}^{\star}\left(s^{\star}, d\right)}{\mathrm{d} t^{s_{i}}}= &\ell_{i 1}(\sigma) u_{i}^{\star}\left(s^{\star}, d\right)+\ell_{i 2}(\sigma) \frac{\mathrm{d} u_{i}^{\star}\left(s^{\star}, d\right)}{\mathrm{d} t}\\
	&+\cdots+\ell_{i s_{i}}(\sigma) \frac{\mathrm{d}^{\left(s_{i}-1\right)} u_{i}^{\star}\left(s^{\star}, d\right)}{\mathrm{d} t^{\left(s_{i}-1\right)}},
\end{aligned}
\end{equation*}
where $\ell_{i 1}(\sigma),\ell_{i 2}(\sigma), \ldots \ell_{i s_{i}}(\sigma)$, $i=1,2, \ldots, N$ are scalars such that the roots of the polynomials $P_{i}^{\sigma}(v)=v^{s_{i}}-\ell_{i 1}(\sigma)-\ell_{i 2}(\sigma) v-\cdots-\ell_{i s_{i}}(\sigma) v^{s_{i}-1}$ are distinct with zero real parts for all $\sigma \in \mathbb{S}$. Let
\begin{align} \label{Phi_Gamma}
\Phi_{i}(\sigma)\!=\! & \left[\!\!\!\begin{array}{cccc}
0 & 1 & \!\cdots\! & 0 \\
\vdots & \vdots & \!\ddots\! & \vdots \\
0 & 0 & \!\cdots\! & 1 \\
\ell_{i 1}(\sigma) & \ell_{i 2}(\sigma) & \!\cdots\! & \ell_{i s_{i}}(\sigma)
\end{array}\!\!\!\right]\!, ~
\Gamma_{i}=\!\!\!\!\! & \left[\begin{array}{cccc}
1 \\ 0 \\ \vdots \\ 0
\end{array}\right]^{\operatorname{T}}\!\!\!.
\end{align}
It then can be verified that $\left(\Gamma_{i}, \Phi_{i}(\sigma)\right),~ i=1,2,\ldots,N $ are observable for all $\sigma \in \mathbb{S}$. Let $\tau_{i}\left(s^{\star}, d\right)=\operatorname{col}\Big(u_{i}^{\star}\left(s^{\star}, d\right), \frac{\mathrm{d} u_{i}^{\star}\left(s^{\star}, d\right)}{\mathrm{d} t}, \ldots, \frac{\mathrm{d}^{\left(s_{i}-1\right)} u_{i}^{\star}\left(s^{\star}, d\right)}{\mathrm{d} t^{\left(s_{i}-1\right)}}\Big)$. One thus has
\begin{align} \label{dot_tau}
\!\!\dot{\tau}_{i}\left(s^{\star}, d\right)=\Phi_{i}(\sigma) \tau_{i}\left(s^{\star}, d\right), \quad 
u_{i}^{\star}\left(s^{\star}, d\right)=\Gamma_{i} \tau_{i}\left(s^{\star}, d\right).
\end{align}

Let $\left(M_{i}, N_{i}\right), i=1,2, \ldots, N$ be any controllable pairs, where $M_{i} \in \mathbb{R}^{s_{i} \times s_{i}}$ is a Hurwitz matrix, and $N_{i} \in \mathbb{R}^{s_{i} \times 1}$ is a column vector. Then the following internal model is proposed,
\begin{equation} \label{internal_model}
\dot{\eta}_{i}=M_{i} \eta_{i}+N_{i} u_{i}, \quad i=1,2, \ldots, N.
\end{equation}
Since for all $\sigma \in \mathbb{S}$, the spectra of $\Phi_{i}(\sigma)$ and $M_{i}$ are disjoint, there exists a nonsingular matrix $T_{i}(\sigma)$ satisfying the Sylvester equation,
\begin{equation} \label{Sylvester_equation}
T_{i}(\sigma) \Phi_{i}(\sigma)-M_{i} T_{i}(\sigma)=N_{i} \Gamma_{i}.
\end{equation}

Let $\theta_{i}\left(s^{\star}, d\right)=T_{i}(\sigma) \tau_{i}\left(s^{\star}, d\right)$ and $\tilde{\eta}_{i}=\eta_{i}-\theta_{i}\left(s^{\star}, d\right)-b_{i}^{-1}(w) N_{i} \vartheta_{i}$. By using (\ref{dot_tau}) and (\ref{Sylvester_equation}), one can deduce that
\begin{align} \label{dot_tilde_eta}
\dot{\tilde{\eta}}_{i} \!= & M_{i} \tilde{\eta}_{i}\!+\!b_{i}^{-\!1}\!(w) M_{i} N_{i} \vartheta_{i}\!-\!b_{i}^{-\!1}\!(w) N_{i} \dot{\vartheta}_{i} \!+\! N_{i}\big(u_{i}\!-\!u_{i}^{\star}\!(s^{\star}\!, d)\big) \notag\\
= & M_{i} \tilde{\eta}_{i}+\tilde{f}_{i 1}\left(\bar{x}_{i 1}, \vartheta_{i}, y_{i}^{r}, d\right) \!-\!b_{i}^{-1}(w) N_{i}\left(\bar{f}_{i 2} \!-\!\gamma \dot{y}_{i}^{r}\right)\!,
\end{align}
where $\tilde{f}_{i 1}\left(\bar{x}_{i 1}, \vartheta_{i}, y_{i}^{r}, d\right)=-b_{i}^{-1}(w) N_{i} \hat{f}_{i}\left(\bar{x}_{i 1}, \vartheta_{i}, y_{i}^{r}, d\right)+b_{i}^{-1}(w) M_{i} N_{i} \vartheta_{i}$ is smooth and satisfies $\tilde{f}_{i 1}\left(0,0, y_{i}^{r}, d\right)=0$ for all $ y_{i}^{r}\in\mathbb{R} $ and $ d\in \mathbb{D} $. Meanwhile, by referring to (\ref{dot_tau}), one has $u_{i}^{\star}\left(s^{\star}, d\right)=\Gamma_{i} \tau_{i}\left(s^{\star}, d\right)=\Gamma_{i} T_{i}^{-1}(\sigma) \theta_{i}\left(s^{\star}, d\right)\triangleq \Psi_{i}(\sigma) \theta_{i}\left(s^{\star}, d\right)$. It then follows that
\begin{align} \label{b_i_u_star}
& b_{i}(w)u_{i}^{\star}\left(s^{\star}, d\right) =b_{i}(w) \Psi_{i}(\sigma) \theta_{i}\left(s^{\star}, d\right) \notag\\
=& -b_{i}(w) \Psi_{i}(\sigma) \tilde{\eta}_{i} +b_{i}(w) \Psi_{i}(\sigma) \eta_{i} -\Psi_{i}(\sigma) N_{i} \vartheta_{i}.
\end{align}
Thus, by using (\ref{error_system_1}) and (\ref{b_i_u_star}), one obtains that
\begin{align} \label{dot_tilde_vartheta}
\dot{\vartheta}_{i} =&\hat{f}_{i}\left(\bar{x}_{i 1}, \vartheta_{i}, y_{i}^{r}, d\right)+b_{i}(w) \Psi_{i}(\sigma) \tilde{\eta}_{i} -b_{i}(w) \Psi_{i}(\sigma) \eta_{i} \notag\\
&+\Psi_{i}(\sigma) N_{i} \vartheta_{i} +b_{i}(w) u_{i} +\bar{f}_{i 2}\left(y_{i}^{r}, s^{\star}, d\right)-\gamma \dot{y}_{i}^{r} \notag\\
\triangleq & \tilde{f}_{i 2}\left(\bar{x}_{i 1}, \tilde{\eta}_{i}, \vartheta_{i}, y_{i}^{r}, d\right)+\bar{f}_{i 2}\left(y_{i}^{r}, s^{\star}, d\right) \notag\\
& -b_{i}(w) \Psi_{i}(\sigma) \eta_{i}+b_{i}(w) u_{i} -\gamma \dot{y}_{i}^{r},
\end{align}
where $\tilde{f}_{i 2}\left(\bar{x}_{i 1}, \tilde{\eta}_{i}, \vartheta_{i}, y_{i}^{r}, d\right) =\hat{f}_{i}\left(\bar{x}_{i 1}, \vartheta_{i}, y_{i}^{r}, d\right)+b_{i}(w) \Psi_{i}(\sigma) \tilde{\eta}_{i}+\Psi_{i}(\sigma) N_{i} \vartheta_{i}$ is a smooth function satisfying $\tilde{f}_{i 2}\left(0,0,0, y_{i}^{r}, d\right)=0$ for all $ y_{i}^{r}\in\mathbb{R} $ and $ d\in \mathbb{D} $. 

Let $ \epsilon_{i}\left(y_{i}^{r},\dot{y}_{i}^{r}, s^{\star}, d\right)= \bar{f}_{i 2}\left(y_{i}^{r}, s^{\star}, d\right) -\gamma \dot{y}_{i}^{r} $. Then by noting the equation (\ref{dot_tilde_eta}) resulting from the internal model (\ref{internal_model}), we can obtain the following augmented error system,
\begin{equation} \label{augmented system 2}
\begin{aligned}
\dot{\bar{x}}_{i 1}=&-\gamma \bar{x}_{i 1}+\vartheta_{i}-\dot{y}_{i}^{r}, \\
\dot{\tilde{\eta}}_{i}=&M_{i} \tilde{\eta}_{i}+\tilde{f}_{i 1}-b_{i}^{-1}(w) N_{i}\epsilon_{i}, \\
\dot{\vartheta}_{i}=&\tilde{f}_{i 2}+b_{i}(w) \big(u_{i} - \Psi_{i}(\sigma) \eta_{i}\big) +\epsilon_{i}.
\end{aligned}
\end{equation}

\textbf{Step 3.} A result on the ISS property of $\bar{x}_{i 1}$- and $ \tilde{\eta}_{i} $-subsystems of (\ref{augmented system 2}) is presented. Denote $ \chi_{i}=\operatorname{col}(\bar{x}_{i1}, \tilde{\eta}_{i}) $. The following lemma describes the ISS property on $ \chi_{i} $-subsystem.

\begin{lemma} \label{lemma_chi_i}
	Consider $ \chi_{i} $-subsystem of (\ref{augmented system 2}). There exists a continuously differentiable function $ V_{i 1}(\chi_{i}) $ such that, for all $ y_{i}^{r}\in\mathbb{R} $ and $ d\in \mathbb{D} $, the following inequalities are satisfied,
	\begin{align*}
	\underline{\alpha}_{i\chi}(\|\chi_{i}\|) \leq & V_{i1}(\chi_{i}) \leq \overline{\alpha}_{i\chi}(\|\chi_{i}\|), \\
	\dot{V}_{i1} \leq & -\chi_{i}^{2} +\gamma_{i\vartheta} \hat{\phi}_{i\vartheta}\left(\vartheta_{i}\right) \vartheta_{i}^{2} \notag\\
	& +\gamma_{ir} \hat{\phi}_{ir}\left(\dot{y}_{i}^{r}\right)\left|\dot{y}_{i}^{r}\right|^{2}  +\gamma_{i\epsilon}|\epsilon_{i}|^{2},
	\end{align*}
	for some known smooth functions $ \underline{\alpha}_{i\chi}, \overline{\alpha}_{i\chi}\in \mathcal{K}_{\infty} $, $ \hat{\phi}_{i\vartheta}, \hat{\phi}_{ir} >1 $, and unknown constants $ \gamma_{i\vartheta}, \gamma_{ir}, \gamma_{i\epsilon}>1 $.	
\end{lemma}
\begin{proof}
	The proof is given in Appendix \ref{Appendix_2}.
\end{proof}

\textbf{Step 4.} An adaptive stabilizer for the augmented system (\ref{augmented system 2}) is developed. More specifically, the decentralized adaptive stabilizer $ u_{i} $ of agent $ i $ is proposed as follows,
\begin{equation} \label{controller_augmented_1}
\begin{aligned}
u_{i} &=-k_{i} \rho_{i}\left(\vartheta_{i}\right) \vartheta_{i}+ \bar{\Psi}_{i} \eta_{i}, \\
\dot{k}_{i} &=\rho_{i}\left(\vartheta_{i}\right) \vartheta_{i}^{2},\quad
\dot{\bar{\Psi}}_{i} =-\eta_{i}^{\operatorname{T}}\vartheta_{i},
\end{aligned}
\end{equation}
where $ \rho_{i} \geq 1 $ is a smooth function to be specified later. Then the closed-loop system composed of the augmented system (\ref{augmented system 2}) and the adaptive stabilizer (\ref{controller_augmented_1}) is given as follows,
\begin{equation} \label{closed-loop system 1}
	\begin{aligned} 
	\dot{\bar{x}}_{i 1}=&-\gamma \bar{x}_{i 1}+\vartheta_{i}-\dot{y}_{i}^{r}, \\
	\dot{\tilde{\eta}}_{i}=&M_{i} \tilde{\eta}_{i}+\tilde{f}_{i 1}-b_{i}^{-1}(w) N_{i}\epsilon_{i}, \\
	\dot{\vartheta}_{i}=&\tilde{f}_{i 2}\!+\!b_{i} \big(\bar{\Psi}_{i}\!-\!\Psi_{i}(\sigma)\big) \eta_{i} \!-\!b_{i} k_{i} \rho_{i}\left(\vartheta_{i}\right) \vartheta_{i} +\epsilon_{i},\\
	\dot{k}_{i} =&\rho_{i}\left(\vartheta_{i}\right) \vartheta_{i}^{2},\quad
	\dot{\bar{\Psi}}_{i} =-\eta_{i}^{\operatorname{T}}\vartheta_{i}.
	\end{aligned}
\end{equation}

By noting Lemma \ref{lemma_chi_i} and applying the changing supply function technique to $ \chi_{i} $-subsystem, given any smooth function $ \Delta_{i\chi}(\chi_{i})>0 $, there exists a continuously differentiable function $\bar{V}_{i1}(\chi_{i})$ such that, along the trajectories of (\ref{augmented system 2}), one has
\begin{align*}
\underline{\alpha}_{i\chi}^{0}&(\|\chi_{i}\|) \leq \bar{V}_{i1}(\chi_{i}) \leq \overline{\alpha}_{i\chi}^{0}(\|\chi_{i}\|), \\
\dot{\bar{V}}_{i1} \leq & -\Delta_{i\chi}(\chi_{i})\chi_{i}^{2} +\gamma_{i\vartheta}^{0} \hat{\phi}_{i\vartheta}^{0}\left(\vartheta_{i}\right) \vartheta_{i}^{2} \notag\\
& +\gamma_{ir}^{0} \hat{\phi}_{ir}^{0}\left(\dot{y}_{i}^{r}\right)\left|\dot{y}_{i}^{r}\right|^{2}  +\gamma_{i\epsilon}^{0}|\epsilon_{i}|^{2} \label{dot_bar_V_1},
\end{align*}
for some known smooth functions $ \underline{\alpha}_{i\chi}^{0}, \overline{\alpha}_{i\chi}^{0}\in \mathcal{K}_{\infty} $, $ \hat{\phi}_{i\vartheta}^{0}, \hat{\phi}_{ir}^{0} >1 $ and unknown constants $ \gamma_{i\vartheta}^{0}, \gamma_{ir}^{0}, \gamma_{i\epsilon}^{0}>1 $.

Define $ \tilde{k}_{i}=k_{i}-k_{0} $ and $ \tilde{\Psi}_{i}= \bar{\Psi}_{i} - \Psi_{i}(\sigma) $ with $ k_{0} $ to be specified later. Consider the following Lyapunov function candidate,
\begin{equation} \label{V_i2_1}
	V_{i2}(\bar{x}_{i1},\tilde{\eta},\vartheta_{i},\tilde{k}_{i},\tilde{\Psi}_{i})\!=\! \bar{V}_{i1}+\frac{1}{2}\vartheta_{i}^{2}+ \frac{1}{2}b_{i}\tilde{k}_{i}^{2} + \frac{1}{2}b_{i}\|\tilde{\Psi}_{i}\|^{2}.
\end{equation}
The derivative of $ V_{i 2} $ along the trajectories of dynamics (\ref{closed-loop system 1}) can be described as follows,
\begin{align} \label{dot_V_i2_1}
\dot{V}_{i 2}=& \dot{\bar{V}}_{i 1}+\vartheta_{i} \dot{\vartheta}_{i}+b_{i}\left(k_{i}-k_{0}\right) \dot{k}_{i}+b_{i} \dot{\bar{\Psi}}\big(\bar{\Psi}_{i} - \Psi_{i}(\sigma)\big)^{\operatorname{T}} \notag\\
=& \dot{\bar{V}}_{i 1}+\vartheta_{i} \tilde{f}_{i 2}-b_{i} k_{0} \rho_{i}\left(\vartheta_{i}\right) \vartheta_{i}^{2}+ \vartheta_{i} \epsilon_{i} \notag\\
\leq &-\!\Delta_{i\chi}(\chi_{i})\chi_{i}^{2} +\gamma_{i\vartheta}^{0} \hat{\phi}_{i\vartheta}^{0}\left(\vartheta_{i}\right) \vartheta_{i}^{2}  +\gamma_{ir}^{0} \hat{\phi}_{ir}^{0}\left(\dot{y}_{i}^{r}\right)\left|\dot{y}_{i}^{r}\right|^{2}  \notag\\
&+\gamma_{i\epsilon}^{0}|\epsilon_{i}|^{2} +\vartheta_{i} \tilde{f}_{i 2}-b_{i} k_{0} \rho_{i}\left(\vartheta_{i}\right) \vartheta_{i}^{2}+ \vartheta_{i} \epsilon_{i}.
\end{align}

Note that $ \tilde{f}_{i 2}\left(\bar{x}_{i 1}, \tilde{\eta}_{i}, \vartheta_{i}, y_{i}^{r}, d\right) $ is sufficiently smooth and satisfies $\tilde{f}_{i 2}\left(0,0,0, y_{i}^{r}, d\right)=0$. By Lemma 11.1 in \cite{chen2015stabilization}, there exist known smooth functions $ \hat{\phi}_{i\chi}, \hat{\phi}_{i\vartheta} >1$ and unknown constant $ \hat{\gamma}_{i1}>1 $ such that, for all $ y_{i}^{r}\in\mathbb{R} $ and $ d\in\mathbb{D} $, the following inequality is satisfied,
\begin{align} \label{tilde_f_i2}
\big\|\tilde{f}_{i 2}\big\|^{2} \leq \hat{\gamma}_{i 1} \left(\hat{\phi}_{i\chi}\left(\chi_{i}\right)\chi_{i}^{2}  +\hat{\phi}_{i\vartheta}\left(\vartheta_{i}\right)\vartheta_{i}^{2}\right).
\end{align}
Meanwhile, the following two inequalities hold,
\begin{align}
	\vartheta_{i} \tilde{f}_{i 2}\leq & \frac{\hat{\gamma}_{i1}}{4}\vartheta_{i}^{2}+ \frac{1}{\hat{\gamma}_{i1}}\|\tilde{f}_{i 2}\|^{2}, \label{vartheta_tilde_f_i2}\\
	\vartheta_{i} \epsilon_{i} \leq & \frac{1}{4}\vartheta_{i}^{2} + |\epsilon_{i}|^{2}. \label{vartheta_epsilon_i}
\end{align}
Then substituting (\ref{tilde_f_i2})-(\ref{vartheta_epsilon_i}) into (\ref{dot_V_i2_1}) yields
\begin{align} \label{dot_V_i2_2}
\dot{V}_{i 2} \leq &-\big(\Delta_{i\chi}(\chi_{i})-\hat{\phi}_{i\chi}\left(\chi_{i}\right)\big)\chi_{i}^{2} -\Big(b_{i} k_{0} \rho_{i}\left(\vartheta_{i}\right) \notag\\
&-\gamma_{i\vartheta}^{0} \hat{\phi}_{i\vartheta}^{0}\left(\vartheta_{i}\right) -\hat{\phi}_{i\vartheta}\left(\vartheta_{i}\right) - \frac{\hat{\gamma}_{i1}+1}{4}\Big) \vartheta_{i}^{2}  \notag\\
&+\gamma_{ir}^{0} \hat{\phi}_{ir}^{0}\left(\dot{y}_{i}^{r}\right)\left|\dot{y}_{i}^{r}\right|^{2} +\big(\gamma_{i\epsilon}^{0}+1\big)|\epsilon_{i}|^{2} .
\end{align}

Furthermore, since $ f_{i} $ is assumed to be sufficiently smooth, it follows from Lemma 3.2 in \cite{khalil2002nonlinear} that $ f_{i}\left(\cdot, \cdot, v, w\right) $ is locally Lipschitz on $ \mathbb{R}^{2} \times \mathbb{R}^{n_{v}} \times \mathbb{R}^{n_{w}} $. Thus, with the result that $ y_{i}^{r} $ is bounded established in Theorem \ref{proposition1}, there exist a smooth function $ l_{i}(\cdot) $ and a constant $ \tilde{l}_{i} $ such that
\begin{align} \label{f_i2}
\big|\bar{f}_{i 2}\left(y_{i}^{r}, s^{\star}, d\right)\big|=&\big|f_{i}\left(y_{i}^{r}, 0, v, w\right)-f_{i}\left(s^{\star}, 0, v, w\right)\big|\notag\\
\leq & l_{i}\big(|\tilde{y}_{i}^{r}|+|s^{\star}|\big)\big|y_{i}^{r}-s^{\star}\big| \notag\\
\leq & \tilde{l}_{i}\big|y_{i}^{r}-s^{\star}\big|= \tilde{l}_{i}|\tilde{y}_{i}^{r}|.
\end{align}
Therefore, by using (\ref{f_i2}), one has
\begin{align} \label{epsilon_i_1}
| \epsilon_{i}\left(y_{i}^{r},\dot{y}_{i}^{r}, s^{\star}, d\right) |^{2}=& | \bar{f}_{i 2}\left(y_{i}^{r}, s^{\star}, d\right) -\gamma \dot{y}_{i}^{r} |^{2} \notag\\
\leq & 2\big| \bar{f}_{i 2}\left(y_{i}^{r}, s^{\star}, d\right) \big|^{2} + 2\gamma^{2}\big|  \dot{y}_{i}^{r} \big|^{2} \notag\\
\leq & 2\tilde{l}_{i}^{2}|\tilde{y}_{i}^{r}|^{2} + 2\gamma^{2}|\dot{y}_{i}^{r}|^{2}.
\end{align}

Then, substituting (\ref{epsilon_i_1}) into (\ref{dot_V_i2_2}) leads to
\begin{align} \label{dot_V_i2_3}
\dot{V}_{i 2} \leq &-\!\big(\Delta_{i\chi}(\chi_{i})\!-\!\hat{\phi}_{i\chi}(\chi_{i})\big)\chi_{i}^{2} -\!\Big(b_{i} k_{0} \rho_{i}\left(\vartheta_{i}\right) \!-\! \frac{\hat{\gamma}_{i1}+1}{4} \notag\\
&-\gamma_{i\vartheta}^{0} \hat{\phi}_{i\vartheta}^{0}\left(\vartheta_{i}\right) -\hat{\phi}_{i\vartheta}\left(\vartheta_{i}\right)\Big) \vartheta_{i}^{2} +2\tilde{l}_{i}^{2}\big(\gamma_{i\epsilon}^{0}+1\big)|\tilde{y}_{i}^{r}|^{2}   \notag\\
&+\big(\gamma_{ir}^{0} \hat{\phi}_{ir}^{0}\left(\dot{y}_{i}^{r}\right)+2\gamma^{2}(\gamma_{i\epsilon}^{0}+1)\big)\left|\dot{y}_{i}^{r}\right|^{2}.
\end{align}
Let $ \Delta_{i\chi}(\chi_{i}) \geq \hat{\phi}_{i\chi}\left(\chi_{i}\right)+1 $, $ k_{0}\geq b_{i}^{-1}\big(\gamma_{i\vartheta}^{0}+\frac{\hat{\gamma}_{i1}+1}{4}+1\big)+1 $, $ \rho_{i}\big(\vartheta_{i}\big) \geq \max\big\{\hat{\phi}_{i\vartheta}^{0}\left(\vartheta_{i}\right) , \hat{\phi}_{i\vartheta}\left(\vartheta_{i}\right)\big\} +1 $, $ \tilde{\gamma}_{i r}\geq 2\tilde{l}_{i}^{2}\big(\gamma_{i\epsilon}^{0}+1\big)+1 $ and $ \hat{\gamma}_{i r} \geq \gamma_{ir}^{0} \hat{\phi}_{ir}^{0}\left(\dot{y}_{i}^{r}\right)+2\gamma^{2}\big(\gamma_{i\epsilon}^{0}+1\big) $. Finally, (\ref{dot_V_i2_3}) can be rewritten as follows,
\begin{align} \label{dot_V_i2_4}
\dot{V}_{i 2} \leq &-\chi_{i}^{2} - \vartheta_{i}^{2} +\tilde{\gamma}_{i r}|\tilde{y}_{i}^{r}|^{2} +\hat{\gamma}_{i r}\left|\dot{y}_{i}^{r}\right|^{2}.
\end{align}

Now we are ready to present the main result of this paper as follows.
\begin{theorem} \label{theorem 2}
	Under Assumptions \ref{assumption_cost functions}--\ref{assumption_u}, the optimal output consensus Problem \ref{problem} for the second-order uncertain nonlinear system (\ref{dynamics}) over a weight-unbalanced directed network is solved by the following distributed dynamic state feedback controller,
	\begin{subequations}\label{controller_1}
	\begin{align} 
	\begin{split} \label{controller_1_1}
	u_{i} &=-k_{i} \rho_{i}\left(\vartheta_{i}\right) \vartheta_{i}+\bar{\Psi}_{i} \eta_{i}, \\[1mm]
	\dot{k}_{i} &=\rho_{i}\left(\vartheta_{i}\right) \vartheta_{i}^{2},\quad
	\dot{\bar{\Psi}}_{i} =-\eta_{i}^{\operatorname{T}}\vartheta_{i}, 
	\end{split}\qquad \qquad \qquad \quad ~~\!\Bigg\}\\[2mm]
	\begin{split} \label{controller_1_2}
	\dot{\eta}_{i} &=M_{i} \eta_{i}+N_{i} u_{i}, 
	\end{split} \qquad \qquad \qquad \quad \qquad \qquad \quad ~~~\!\Big\}\\
	\begin{split} \label{controller_1_3}
	\dot{y}_{i}^{r} & =-\frac{1}{\xi_{i}^{i}} \nabla c_{i}\left(y_{i}^{r}\right)-\beta_{1} \sum_{i=1}^{N} a_{i j}\left(y_{i}^{r}-y_{j}^{r}\right)-\beta_{2} z_{i},  \\[-1.5mm]
	\dot{z}_{i} & =\beta_{1} \sum_{i=1}^{N} a_{i j}\left(y_{i}^{r}-y_{j}^{r}\right), \quad z_{i}(0)=0, \\[-1mm]
	\dot{\xi}_{i} &= -\sum_{j=1}^{N}a_{ij}(\xi_{i}-\xi_{j}),
	\end{split} ~\Biggggr\}
	\end{align}
	\end{subequations}
	where $ \rho_{i}(\cdot) \geq 1 $ is chosen to be the same as in (\ref{dot_V_i2_3}),  $\vartheta_{i}=x_{i 2}+\gamma (x_{i 1}- y_{i}^{r} )$ with $ \gamma \geq \frac{3}{2} $ as in (\ref{dot_V_ix_1}), the pair $ (M_{i}, N_{i}) $ is selected as in (\ref{internal_model}), and $ \beta_{1}, \beta_{2} >0 $ are sufficiently large constants satisfying (\ref{paremeters_beta_i}).
\end{theorem}

\begin{proof}
	By Lemma \ref{lemma_stabilization} and the controller design procedure, it is sufficient to prove that the distributed dynamic controller (\ref{controller_1_1}) and (\ref{controller_1_2}) solves Problem \ref{problem_2}, that is, stabilizes the augmented system composed of the error system (\ref{error_system}) and the optimal coordinator (\ref{controller_1_3}). By substituting the distributed dynamic controller (\ref{controller_1}) into the agent dynamics (\ref{dynamics}), one is able to obtain the following closed-loop system,
	\begin{subequations}\label{closed-loop system_1}
		\begin{align} 
		\begin{split} \label{closed-loop system_1_1}
		\dot{\bar{x}}_{i 1}=&-\gamma \bar{x}_{i 1}+\vartheta_{i}-\dot{y}_{i}^{r}, \\
		\dot{\tilde{\eta}}_{i}=&M_{i} \tilde{\eta}_{i}+\tilde{f}_{i 1}-b_{i}^{-1}(w) N_{i}\epsilon_{i}, \\
		\dot{\vartheta}_{i}=&\tilde{f}_{i 2}\!+\!b_{i} \big(\bar{\Psi}_{i}\!-\!\Psi_{i}(\sigma)\big) \eta_{i} \!-\!b_{i} k_{i} \rho_{i}\left(\vartheta_{i}\right) \vartheta_{i} +\epsilon_{i},\\
		\dot{k}_{i} =&\rho_{i}\left(\vartheta_{i}\right) \vartheta_{i}^{2},\quad
		\dot{\bar{\Psi}}_{i} =-\eta_{i}^{\operatorname{T}}\vartheta_{i}, 
		\end{split}~~~~\!\! \biggggr\}\\[1mm]
		\begin{split} \label{closed-loop system_1_2}
		\dot{y}_{i}^{r} = &-\frac{1}{\xi_{i}^{i}} \nabla c_{i}\left(y_{i}^{r}\right)-\beta_{1} \sum_{i=1}^{N} a_{i j}\left(y_{i}^{r}-y_{j}^{r}\right)-\beta_{2} z_{i},  \\[-2mm]
		\dot{z}_{i} =&\beta_{1} \sum_{i=1}^{N} a_{i j}\left(y_{i}^{r}-y_{j}^{r}\right), \quad z_{i}(0)=0,\\[-1mm]
		\dot{\xi}_{i} =& -\sum_{j=1}^{N}a_{ij}(\xi_{i}-\xi_{j}).
		\end{split}~\!\Biggggr\}
		\end{align}
	\end{subequations}
	
	We are now ready to prove the asymptotical stability of the closed-loop system (\ref{closed-loop system_1}) at $ (\bar{x}_{i 1}, \vartheta_{i})=(0, 0) $. To this end, reconsider the Lyapunov function candidate $ V_{i2} $ in (\ref{V_i2_1}). Note that the subsystem (\ref{closed-loop system_1_1}) is the same as (\ref{closed-loop system 1}). By referring to (\ref{dot_V_i2_4}), one thus has
	\begin{align} \label{dot_V_i2_5}
	\dot{V}_{i 2}|_{(\ref{closed-loop system_1})} = \dot{V}_{i 2}|_{(\ref{closed-loop system 1})} \leq &-\!\chi_{i}^{2} - \vartheta_{i}^{2} +\tilde{\gamma}_{i r}|\tilde{y}_{i}^{r}|^{2} +\hat{\gamma}_{i r}\left|\dot{y}_{i}^{r}\right|^{2}.
	\end{align}

	It is shown in the proof of Theorem \ref{proposition1} that the origin $ (\tilde{y}^{r},\tilde{z})=(0,0) $ is an exponentially stable equilibrium point of the system (\ref{algorithm_r}). By applying Theorem 4.14 in \cite{khalil2002nonlinear}, there exists a Lyapunov function $ \bar{V}_{0}(\tilde{y}^{r},\tilde{z}) $ that satisfies the following inequalities
	\begin{align}
		\bar{c}_{1}\big(\left\|\tilde{y}^{r}\right\|^{2}+\|\tilde{z}\|^{2}\big) \leq& \bar{V}_{0} \leq \bar{c}_{2}\big(\left\|\tilde{y}^{r}\right\|^{2}+\|\tilde{z}\|^{2}\big), \\
		\dot{\bar{V}}_{0}|_{(\ref{algorithm_r})} \leq&-\bar{c}_{3}\big(\left\|\tilde{y}^{r}\right\|^{2}+\|\tilde{z}\|^{2}\big), 
	\end{align}
	for some positive constants $ \bar{c}_{1} $, $ \bar{c}_{2} $ and $ \bar{c}_{3} $.
	Similarly, since (\ref{closed-loop system_1_2}) coincides with (\ref{algorithm_r}), by adopting the same Lyapunov function candidate $ \bar{V}_{0}(\tilde{y}^{r},\tilde{z}) $, one has 
	\begin{equation} \label{dot_V_0_final_1}
	\dot{\bar{V}}_{0}|_{(\ref{closed-loop system_1})} = \dot{\bar{V}}_{0}|_{(\ref{algorithm_r})} \leq -\bar{c}_{3}\big(\left\|\tilde{y}^{r}\right\|^{2}+\|\tilde{z}\|^{2}\big).
	\end{equation}
	
	Now, consider the Lyapunov function candidate $ V=\sum_{i=1}^{N}V_{i2} + \mu \bar{V}_{0} $, where $ \mu > 0 $ is a constant to be determined later. Note that $ \chi = \operatorname{col}(\chi_{1},\chi_{2},\ldots,\chi_{N}) $ and $ \vartheta = \operatorname{col}(\vartheta_{1},\vartheta_{2},\ldots,\vartheta_{N}) $. By combining (\ref{dot_V_i2_5}) and (\ref{dot_V_0_final_1}), the derivative of $ V $ along the trajectories of the closed-loop system (\ref{closed-loop system_1}) satisfies
	\begin{align} \label{dot_V_1}
	\dot{V} \leq &-\!\chi^{2} - \vartheta^{2} +\sum_{i=1}^{N}\tilde{\gamma}_{i r}|\tilde{y}_{i}^{r}|^{2} +\sum_{i=1}^{N}\hat{\gamma}_{i r}\left|\dot{y}_{i}^{r}\right|^{2} \notag\\
	&- \mu \bar{c}_{3}\big(\left\|\tilde{y}^{r}\right\|^{2}+\|\tilde{z}\|^{2}\big).
	\end{align}
	
	Under Assumption \ref{assumption_cost functions}, it is worth noting that the function $ \nabla \tilde{c}\left(y^{r}\right) -\nabla \tilde{c}\left(\bar{y}^{r}\right) $ is Lipschitz with respect to $ \tilde{y}^{r}= y^{r}-\bar{y}^{r} $. Thus, it follows from (\ref{compact form_r_a}) and (\ref{equilibrium point_r_a}) that $ \dot{y}^{r} $ is Lipschitz in $ \operatorname{col}(\tilde{y}^{r}, \tilde{z}) $. Thus, there exists a constant $ \mu_{1}>0 $ such that
	\begin{align} \label{hat_gamma_ir}
	\sum_{i=1}^{N}\hat{\gamma}_{i r}\left|\dot{y}_{i}^{r}\right|^{2} \leq \hat{\gamma}_{r}\|\dot{y}^{r}\|^{2} \leq \mu_{1}\bar{c}_{3}\big(\left\|\tilde{y}^{r}\right\|^{2}+\|\tilde{z}\|^{2}\big),
	\end{align}
	where $ \hat{\gamma}_{r} = \max \big\{ \hat{\gamma}_{i r}, i=1,2,\ldots,N \big\}$. Meanwhile, it holds that
	\begin{align} \label{tilde_gamma_ir}
	\sum_{i=1}^{N}\tilde{\gamma}_{i r}|\tilde{y}_{i}^{r}|^{2} \leq \tilde{\gamma}_{r}\|\tilde{y}^{r}\|^{2} \leq \tilde{\gamma}_{r}\big(\left\|\tilde{y}^{r}\right\|^{2}+\|\tilde{z}\|^{2}\big),
	\end{align}
	where $ \tilde{\gamma}_{r} = \max \big\{ \tilde{\gamma}_{i r}, i=1,2,\ldots,N \big\}$. Let $ \mu\geq \mu_{1} + \frac{\tilde{\gamma}_{r}}{\bar{c}_{3}} $. Then, substituting (\ref{hat_gamma_ir}) and (\ref{tilde_gamma_ir}) into (\ref{dot_V_1}) yields
	\begin{align*} \label{dot_V_2}
	\dot{V} \leq &- \chi^{2} - \vartheta^{2}.
	\end{align*}
	This indicates that the origin is an asymptotically stable equilibrium point of the closed-loop system (\ref{closed-loop system_1_1}). Therefore, the reference-tracking problem, that is, Problem \ref{problem_2} for the error system (\ref{error_system}) is solved. One concludes that the optimal output consensus problem for the second-order uncertain nonlinear multi-agent system (\ref{dynamics}) over a weight-unbalanced directed network, that is, Problem \ref{problem}, is solved. The proof is thus completed.
\end{proof}

\begin{remark} \label{remark_controller design}
	The distributed dynamic state feedback controller (\ref{controller_1}) consists of three components, namely, the distributed optimal coordinator (\ref{controller_1_3}), the internal model (\ref{controller_1_2}) and the decentralized adaptive controller (\ref{controller_1_1}).
\end{remark}

\begin{remark} 
	It is worth pointing out that the controllers developed in \cite{li2021distributed,li2018consensus,li2020distributed} take linear combination of information between neighbors with constant control gains. Thus, the nonlinear functions considered in the above-mentioned works are required to be globally Lipschitz so that the positive terms resulting from the nonlinear functions can be dominated by choosing sufficiently large control gains. On the contrary, the controller (\ref{controller_1}) proposed in this work is of a nonlinear form, which is able to deal with functional coefficients of positive terms. Therefore, by utilizing Lemma 11.1 in \cite{chen2015stabilization}, the restrictive globally Lipschitz conditions on nonlinear agent dynamics can be avoided.
\end{remark}
 
In what follows, we study the conditions under which the estimated parameter vectors $ \bar{\Psi}_{i}(t), i=1,2,\ldots,N $ converge to the theoretical parameter vectors $ \Psi_{i}(\sigma) $. Under Assumption \ref{assumption_u}, it is shown in \cite{liu2009parameter} that there exists an integer $ s_{i} $ such that, along all the trajectories $ v(t) $ and all $ \operatorname{col}(w,\sigma)\in \mathbb{W}\times\mathbb{S} $,
\begin{equation}
	u_{i}^{\star}(s^{\star},v, w, \sigma)=\sum_{l=1}^{s_{i}} C_{i l}\left(s^{\star}, v_{0}, w, \sigma\right) e^{j\hat{\omega}_{i l} t},
\end{equation}
where $ j=\sqrt{-1} $,~ $ \hat{\omega}_{i 1},\hat{\omega}_{i 2}, \ldots, \hat{\omega}_{i s_{i}}$ are distinct and determined by the eigenvalues of $S(\sigma)$, and $C_{i l}\left(s^{\star}, v_{0}, w, \sigma\right) \in \mathbb{C}$ are not identically zero for all $\operatorname{col}\left(v_{0}, w, \sigma\right)$. We have the following result.

\begin{theorem} \label{theorem Psi}
	Under Assumptions 1-5, suppose that the internal model (\ref{internal_model}) is of minimal order, and none of $C_{i l}\left(s^{\star}, v_{0}, w, \sigma\right)$ is identically zero for all $ v(t) $ and all $ \operatorname{col}(w,\sigma)\in\mathbb{W}\times\mathbb{S} $. Then under the distributed controller $(\ref{controller_1})$, one has $\lim _{t \rightarrow \infty}\bar{\Psi}_{i}(t)=\Psi_{i}(\sigma), ~i=1,2, \ldots, N$.
\end{theorem}
\begin{proof}
	The proof follows similar arguments as those in the proof of Theorem 2 in \cite{su2013cooperative}, and is thus omitted.
\end{proof}

\section{Illustrative Examples} \label{section simulation results}

In this section, two simulation examples are provided to illustrate the effectiveness of the proposed distributed adaptive controller. 

\subsection{Example 1}
Consider a collection of five agents whose dynamics are described by the following second-order uncertain nonlinear systems,
\begin{equation} \label{Van der Pol}
	\begin{aligned}
		\dot{x}_{i 1}\!= & x_{i 2}, \\
		\dot{x}_{i 2}\!= & \!-\!x_{i 1}x_{i 2}\!+\!\mu_{i 1}(w) x_{i 2}\left(1\!-\!x_{i 1}^{2}\right)\!+\!b_{i}(w) u_{i} \!+\! A_{w}\sin(\sigma t), \\
		y_{i}\!= & x_{i 1}, \quad i=1,2,\ldots,5,
	\end{aligned}
\end{equation}
where $ A_{w}\sin(\sigma t) $ is an external disturbance with $ A_{w} = \mu_{i 2}(w)A  $ being its uncertain amplitude and $ \sigma $ being its angular frequency. It can be verified that $ A_{w}\sin(\sigma t)= A_{w}v_{1}$ with $ v $ being generated by the exosystem (\ref{exosystem}) with $S(\sigma)=\Big[\begin{array}{cc}0 & \sigma \\ -\sigma & 0\end{array}\Big]$. Thus, it is known that Assumption \ref{assumption_exosysem} is satisfied. Assume that the uncertain coefficients $ \mu_{i} = \big(\mu_{i 1}(w), \mu_{i 2}(w), b_{i}(w)\big) $ satisfy $ \mu_{i} = \bar{\mu}_{i} + w_{i} $, where $ \bar{\mu}_{i}= (\bar{\mu}_{i 1}, \bar{\mu}_{i 2}, \bar{b}_{i}) $ denotes the nominal value of $ \mu_{i} $, $ w_{i} = (w_{i1}, w_{i2}, w_{i3}) \in \mathbb{W} $ denotes the uncertainty with $ \mathbb{W} $ being an unknown compact set. It is worth pointing out that the nonlinear function on the right hand side of (\ref{Van der Pol}) is not required to be globally Lipschitz.

The unbalanced directed communication topology among agents is described in Fig. \ref{Fig_topology1}. It can be verified that the directed graph is strongly connected. For $ i=1,2,\ldots,5 $, suppose that each agent $ i $ possesses a local cost function $ c_{i}(s)=0.1(s-i)^{2} $. It can be verified that all local cost functions are strongly convex, and the global minimizer is $ s^{\star}=3 $. Therefore, Assumptions \ref{assumption_cost functions} and \ref{graph assumption} are fulfilled.

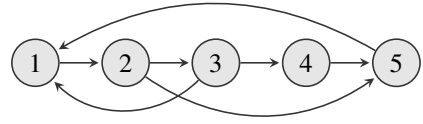
\begin{figure}[!t] 
	\centering 
	\begin{tikzpicture}[> = stealth, 
	shorten > = 1pt, 
	auto,
	node distance = 3cm, 
	semithick 
	,scale=0.6,auto=left,every node/.style={circle,fill=gray!20,draw=black!80,text centered}]
	\centering
	\node (n1) at (0,0)		{1};
	\node (n2) at (2,0)  	{2};
	\node (n3) at (4,0) 	{3};
	\node (n4) at (6,0) 	{4};
	\node (n5) at (8,0) 	{5};
	
	\draw[->,black!80] (n3) to [out=-135,in=-45] (n1);
	\draw[->,black!80] (n1)-- (n2);
	\draw[->,black!80] (n2)-- (n3);
	\draw[->,black!80] (n3)-- (n4);
	\draw[->,black!80] (n4)-- (n5);
	\draw[->,black!80] (n2) to [out=-35,in=-145] (n5);
	\draw[->,black!80] (n5) to [out=150,in=30] (n1);
	
	\end{tikzpicture} 
	\caption{Weight-unbalanced directed network.} 
	\label{Fig_topology1}
\end{figure}

For $ i=1,2,\ldots,5 $, it is calculated that $u_{i}^{\star}\left(s^{\star}, d\right)=-b_{i}(w)^{-1} A_{w} v_{1}$. Thus, Assumption \ref{assumption_u} is also satisfied. By applying Theorem \ref{theorem 2}, the optimal output consensus problem for uncertain system (\ref{Van der Pol}) over the unbalanced directed network depicted in Fig. \ref{Fig_topology1} is solvable. To proceed, it can be deduced that  $\frac{\mathrm{d}^{2} u_{i}^{\star}\left(s^{\star}, d\right)}{\mathrm{d} t^{2}}=-\sigma^{2} u_{i}^{\star}\left(s^{\star}, d\right)$. It then follows from (\ref{Phi_Gamma}) that $\Phi_{i}(\sigma)=\Big[\begin{array}{cc}0 & 1 \\ -\sigma^{2} & 0\end{array}\Big], ~ \Gamma_{i}=[\begin{array}{ll}1 & 0\end{array}]$. Let $ (M_{i}, N_{i}) $, $ i=1,2,\ldots,5 $ be any controllable pairs of the form $M_{i}=\Big[\begin{array}{cc}0 & 1 \\ -\omega_{1} & -\omega_{2}\end{array}\Big], ~ N_{i}=[\begin{array}{ll}0 & 1\end{array}]^{\operatorname{T}}$, where $ \omega_{1}, \omega_{2} > 0 $ are two constants to be specified later. With these choices, it is noted that the internal model (\ref{internal_model}) is of minimal order. By solving the Sylvester equation (\ref{Sylvester_equation}), one further obtains that
\begin{align*}
	T^{-1}(\sigma)=&\left[\begin{array}{cc}
		\omega_{1}-\sigma^{2} & \omega_{2} \\
		-\omega_{2} \sigma^{2} & \omega_{1}-\sigma^{2}
	\end{array}\right], \\
	\Psi_{i}(\sigma)=&\Gamma_{i} T^{-1}(\sigma)=\left[\omega_{1}-\sigma^{2} \quad \omega_{2}\right].
\end{align*} 

In the simulation, set $ A=10 $ and $ \sigma=0.8 $. For $ i=1,2,\ldots,5 $, let $ \bar{\mu}_{i 1}=\bar{\mu}_{i 2} = i $, $\bar{b}_{i}=1$, and $ w_{i} = (w_{i1}, w_{i2}, w_{i3}) $ be randomly generated such that $ \mu_{i 1} $ and $ b_i $ are positive. Besides, choose $ \omega_{1}=2 $ and $ \omega_{2}=3 $. Then, it can be calculated that $ \Psi_{i}|_{\sigma=0.8}=[1.36, ~ 3] $. The other initial conditions are randomly chosen. Then applying the proposed distributed adaptive controller (\ref{controller_1}) to the second-order uncertain nonlinear systems (\ref{Van der Pol}), with $ \rho_{i}(\vartheta_{i})= \vartheta_{i}^{4}+1 $ and $ \vartheta_{i}=2(x_{i1}-y_{i}^{r})+x_{i2} $. The simulation results are presented in Fig. \ref{Fig_second_order_example1_1}.

\begin{figure}[t] 
	\begin{center}
		\includegraphics[height=10.8cm, width=8.8cm]{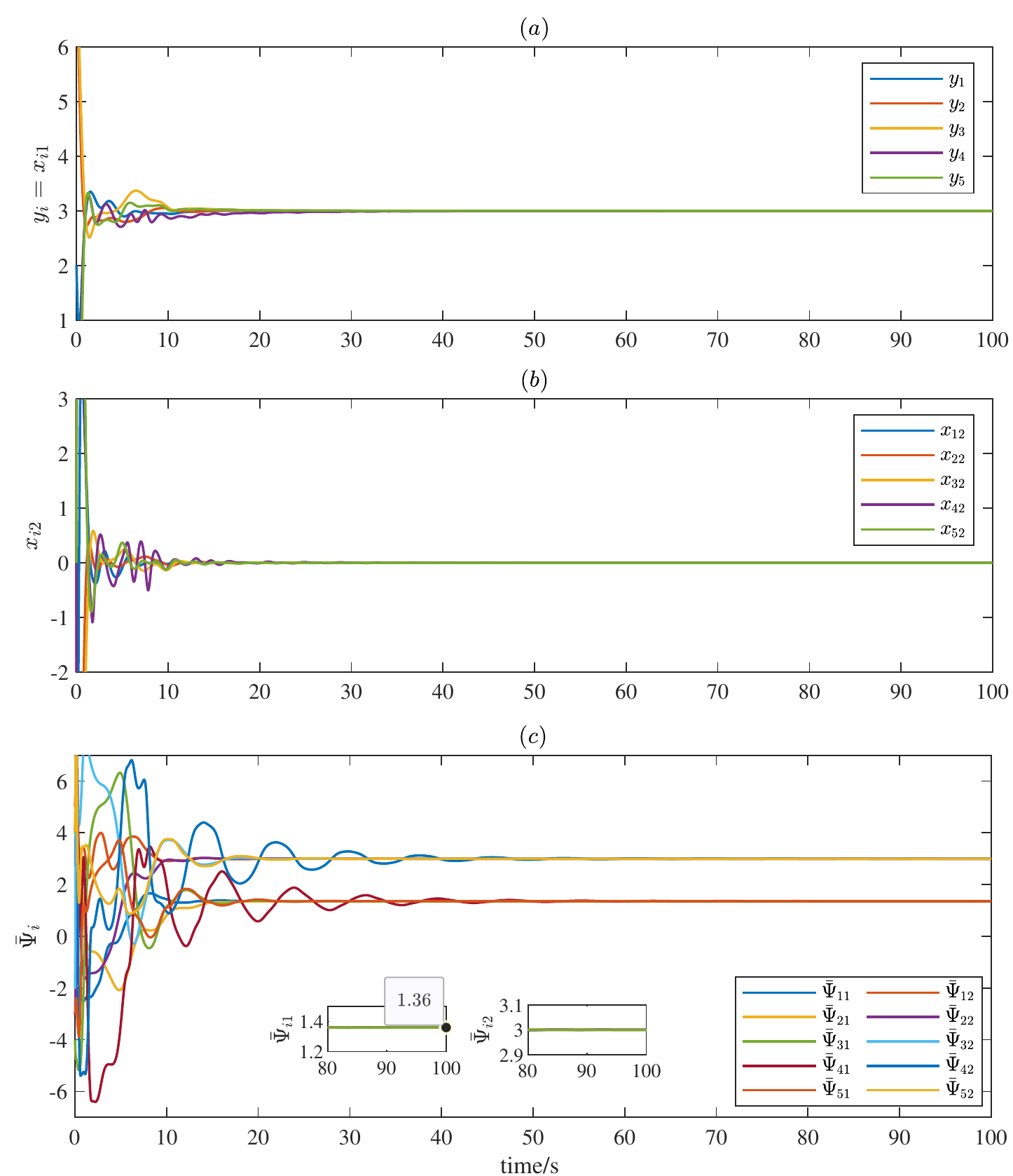} 
		\caption{Convergence performances of the second-order uncertain nonlinear systems (\ref{Van der Pol}) under the distributed dynamic controller (\ref{controller_1}). (a) Trajectories of agent outputs $ y_{i} $'s; (b) Trajectories of agent states $ x_{i2} $'s; (c) Trajectories of adaptive gain vectors $ \bar{\Psi}_{i}=[\bar{\Psi}_{i1},~\bar{\Psi}_{i2}] $'s.}  
		\label{Fig_second_order_example1_1} 
	\end{center}
\end{figure}

Fig. \ref{Fig_second_order_example1_1}(a) and (b) present the convergence performances of agent outputs $ y_{i} $ and corresponding states $ x_{i2} $, $ i=1,2,\ldots,5 $, respectively. It is observed from Fig. \ref{Fig_second_order_example1_1}(a) that the outputs of all the agents eventually approach the optimal value $ s^{\star}=3 $. Meanwhile, it can be observed from Fig. \ref{Fig_second_order_example1_1}(b) that the trajectories of agent states $ x_{i2}, i=1,2,\ldots,5 $ converge to the origin. The convergence performances of the adaptive gain vectors $ \bar{\Psi}_{i} $ are presented in \ref{Fig_second_order_example1_1}(c). It is seen that $ \bar{\Psi}_{i1} $ and $ \bar{\Psi}_{i2} $ tend to 1.36 and 3 respectively, which is consistent with the theoretical result in Theorem \ref{theorem Psi}.

\begin{figure}[t] 
	\centering 
	\includegraphics[height=7.6cm, width=8.8cm]{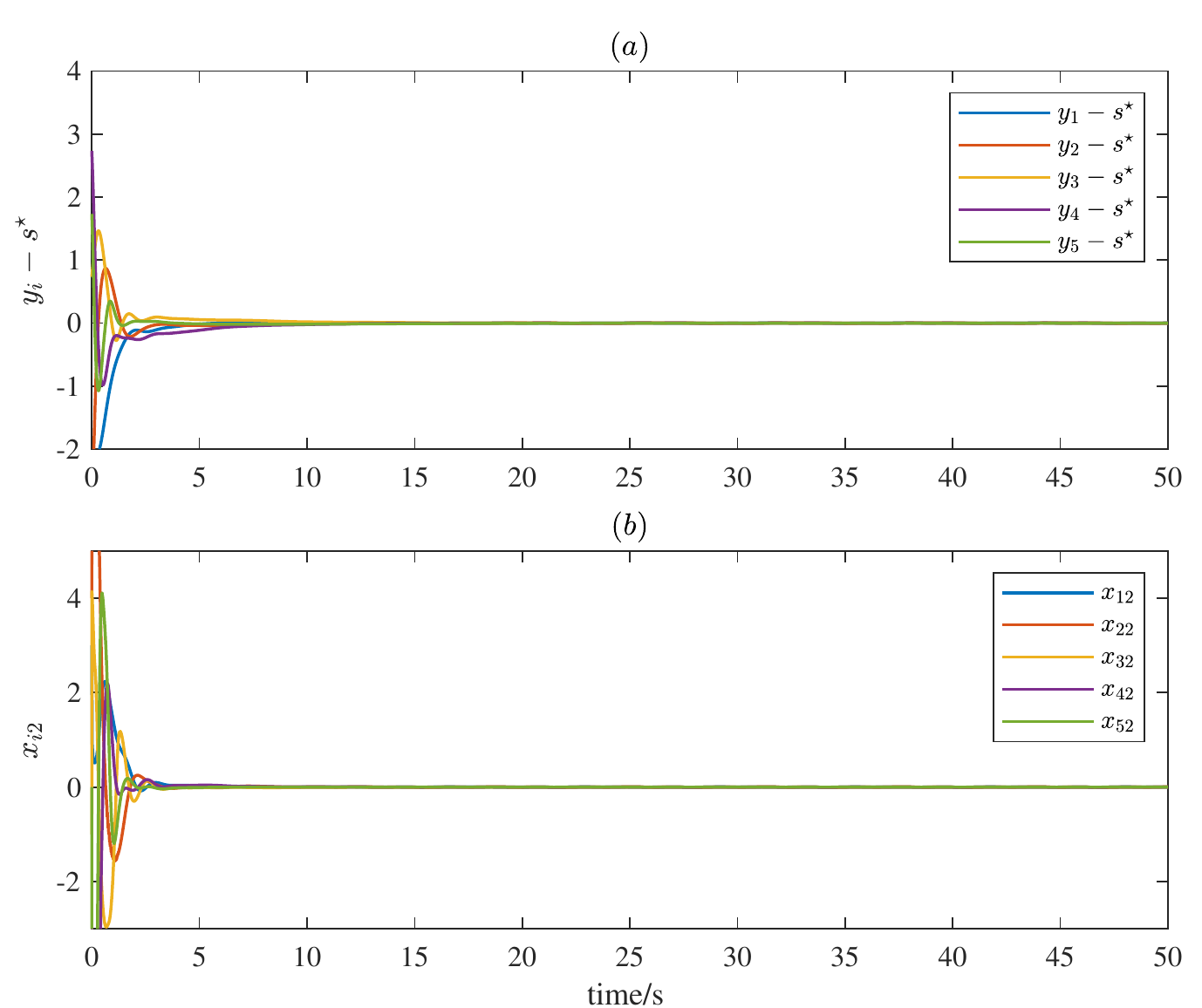} 
	\caption{Convergence performances of the uncertain nonlinear damping-spring system (\ref{nonlinear damping type}) under the distributed dynamic controller (\ref{controller_1}). (a) Trajectories of agent outputs $ y_{i} $'s; (b) Trajectories of agent states $ x_{i2} $'s.}
	\label{Fig_second_order_example2_1}
\end{figure}

\begin{figure}[t] 
	\centering 
	\includegraphics[height=7.6cm, width=8.95cm]{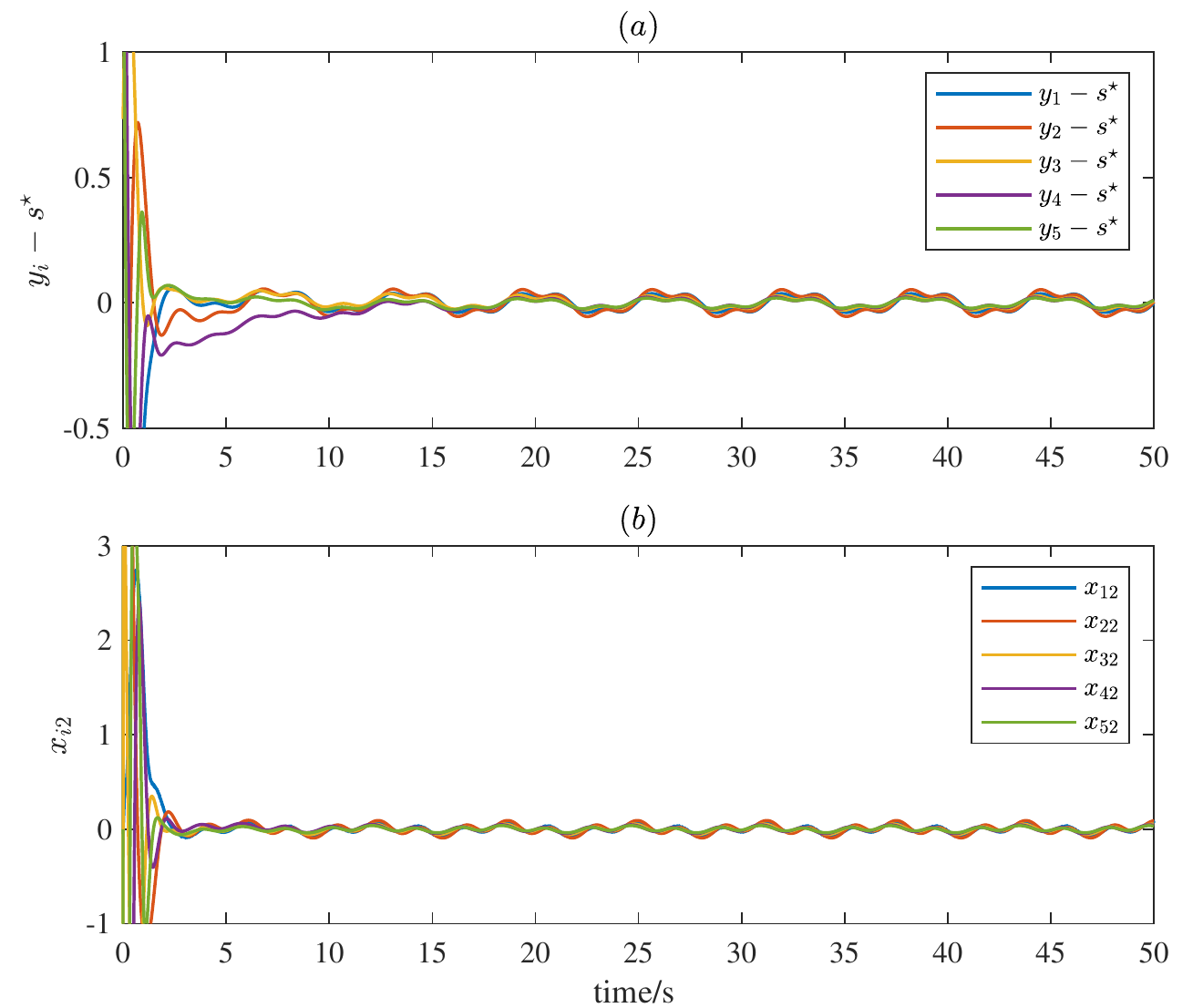} 
	\caption{Convergence performances of the uncertain nonlinear damping-spring system (\ref{nonlinear damping type}) under the distributed dynamic controller (\ref{controller_1}) without an internal model. (a) Trajectories of agent outputs $ y_{i} $'s; (b) Trajectories of agent states $ x_{i2} $'s.}
	\label{Fig_second_order_example2_2}
\end{figure}

\subsection{Example 2}

Consider a more practical scenario where the dynamics of agent $ i $, $ i=1,2,\ldots,5, $ are of the uncertain nonlinear damping-spring type \cite{zhang2009suppressing,wang2018nonlinear},
\begin{equation} \label{nonlinear damping type}
	m_{i} \ddot{y}_{i}+\kappa_{i1} y_{i}+\kappa_{i2} y_{i}^{3}+\mu_{i1} \dot{y}_{i}+\mu_{i2} \dot{y}_{i}^{3} +d_{i}(v,w)=u_{i}, 
\end{equation} 
where $y_{i}$ represents the position of agent $ i $, $\dot{y}_{i}$ represents its velocity, $u_{i}$ is its control, $ d_{i}(v,w)= A_{w}v_{2}(1-v_{1}^{2}) $ is the external disturbance with $ v $ being generated by the exosystem (\ref{exosystem}) with $S(\sigma)=\Big[\begin{array}{cc}0 & 1 \\ -1 & 0\end{array}\Big]$ and $ A_{w} $ being an uncertain constant. $m_{i}$ denotes the uncertain mass, $\kappa_{i1}, \kappa_{i2}$ are uncertain spring constants, and $\mu_{i1}, \mu_{i2}$ are uncertain damping constants.

Let $ x_{i1}=y_{i} $ and $ x_{i2}=\dot{y}_{i} $. Then (\ref{nonlinear damping type}) can be rewritten in the form of (\ref{dynamics}) with $ f_{i}(x_{i1},x_{i2},v,w)= -m_{i}^{-1}\big( \kappa_{i1} x_{i1}+\kappa_{i2} x_{i1}^{3}+ \mu_{i1} x_{i2}+\mu_{i2} x_{i2}^{3} + d_{i}(v,w) \big) $ and $ b_{i}(w)= -m_{i}^{-1} $. Reconsider the unbalanced directed communication topology depicted in Fig. \ref{Fig_topology1}. In this example, we consider more general local cost functions as follows,
\begin{align*}
	f_{1} & =0.25e^{-0.2s}+0.5e^{0.5s}, \quad f_{2}=0.5(s-2)^{2}+e^{0.1s}, \\
	f_{3} & =0.2s\ln(1+s^{2})+s^{2}, \quad f_{4}=0.4\frac{s}{\sqrt{1+s^{2}}} + 0.5s^{2} ,\\
	f_{5} & =0.6s^{2}(\ln(s^{2}+0.5)+1) + 0.3s^{2}/\sqrt{s^{2}+5}.
\end{align*}
It can be verified that all local cost functions are strongly convex. Therefore, Assumptions \ref{assumption_cost functions} and \ref{graph assumption} are satisfied. It is calculated that the global solution takes an approximate value $ s^{\star}=0.267 $.

In the simulation, let $ \big(m_{i},\kappa_{i1},\kappa_{i2},\mu_{i 1},\mu_{i 2}\big)=(1+0.1i,2+0.2i,3-0.1i,4-0.2i,5-0.3i),~ i=1,2,\ldots,5 $, $ A_{w}=100 $. Note that $ u_{i}^{\star}(s^{\star},d)=-\big(\kappa_{i1}s^{\star}+\kappa_{i2}s^{\star 3}+A_{w}v_{2}(1-v_{1}^{2})\big) $. Thus, Assumptions \ref{assumption_exosysem} and \ref{assumption_u} are fulfilled. Choose the controllable pairs $ (M_{i}, N_{i}) $, $ i=1,2,\ldots,5 $ as follows
\begin{equation*}
	M_{i}=\left[\begin{array}{c|c}
		0 & I_{3} \\
		\hline-10  & -18,-15 , -6
	\end{array}\right], \quad N_{i}=[0,0,0,1]^{\operatorname{T}}.
\end{equation*}  
Similarily, by Theorem \ref{theorem 2}, the distributed controller is of the form (\ref{controller_1}) with $ \rho_{i}(\vartheta_{i})= \vartheta_{i}^{4}+1 $ and $ \vartheta_{i}=2(x_{i1}-y_{i}^{r})+x_{i2} $. Set $ z_{i}(0)=0 $, while the other initial conditions are randomly chosen. The simulation results are presented in Fig. \ref{Fig_second_order_example2_1}.

It is observed from Fig. \ref{Fig_second_order_example2_1}(a) that $ y_{i}-s^{\star} $, $ i=1,2,\ldots,5 $ converge to the origin. In other words, the trajectories of agent outputs $ y_{i}, i=1,2,\ldots,5 $ eventualy reach the optimal solution $ s^{\star}=0.267 $. It can be observered from Fig. \ref{Fig_second_order_example2_1}(b) that all the states $ x_{i2} $, $ i=1,2,\ldots,5 $ tend to the origin before 50s. Therefore, it is shown that the distributed adaptive controller (\ref{controller_1}) can be applied to (\ref{nonlinear damping type}) to solve the distributed optimal output consensus problem over the unbalanced directed network.

The convergence performances in Fig. \ref{Fig_second_order_example2_1} show that the designed controller with an internal model (\ref{internal_model}) is effective in dealing with exogenous disturbances in agent dynamics. As a comparison, the convergence performances of the uncertain nonlinear damping-spring system (\ref{nonlinear damping type}) under the distributed dynamic controller (\ref{controller_1}) without an internal model is illustrated in Fig. \ref{Fig_second_order_example2_2}. It is observed from Fig. \ref{Fig_second_order_example2_2} that the distributed dynamic controller without an internal model can guarantee the agent outputs converge to the neighborhood of the optimal solution with a bounded error but cannot reach the exact value.

\section{Conclusion} \label{section conclusion}
This paper develops a new distributed adaptive controller to deal with the distributed optimal output consensus problem for second-order uncertain nonlinear multi-agent systems over unbalanced directed networks. The controller is designed based on a two-layer strategy and does not need any prior information of the dynamics uncertainties. It is shown that the proposed controller is able to steer all the agent outputs to the optimal solution which minimizes the global cost function. Two simulation examples are also provided to illustrate the effectiveness of the control scheme. Future work will concentrate on the distributed optimal output consensus problem over unbalanced directed networks for more general agent systems with only output information.


\begin{appendices}

\section{Proof of Theorem \ref{proposition1}} \label{Appendix_1}

Define $ y^{r}=\operatorname{col}(y_{1}^{r},y_{2}^{r},\ldots,y_{N}^{r}) $, $ z=\operatorname{col}(z_{1},z_{2},\ldots,z_{N}) $, $ \Xi^{-1}=\operatorname{diag}(\frac{1}{\xi_{1}^{1}},\frac{1}{\xi_{2}^{2}},\ldots,\frac{1}{\xi_{N}^{N}}) $, $ R^{-1}=\operatorname{diag}(\frac{1}{\varrho_{1}},\frac{1}{\varrho_{2}},\ldots,\frac{1}{\varrho_{N}}) $ and $ \nabla\tilde{c}(y^{r})$$=$$\operatorname{col}\big(\nabla c_{1}(y_{1}^{r}),\nabla c_{2}(y_{2}^{r}),\ldots,\nabla c_{N}(y_{N}^{r})\big) $. Note that $ \xi_{i}^{i}, i=1,2,\ldots,N $ are shown to be positive. It follows that the matrix $ \Xi^{-1} $ is well defined. Then, the optimal coordinator (\ref{algorithm_r}) can be rewritten in the following compact form,
\begin{subequations} \label{compact form_r}
	\begin{align}  
	\dot{y}^{r}= & -\Xi^{-1} \nabla \tilde{c}\left(y^{r}\right)-\beta_{1} \mathcal{L} y^{r}-\beta_{2} z, \label{compact form_r_a}\\
	\dot{z}= & \beta_{1} \mathcal{L} y^{r}	, \quad z(0)=0, \label{compact form_r_b}\\
	\dot{\xi} = & -\mathcal{L} \xi.
	\end{align}
\end{subequations}

To facilitate the analysis, define $ Y=\operatorname{col}(y^{r}, z) $. Then the dynamics of $ Y $ can be rewritten as
\begin{equation}
	\dot{Y}=\varphi(Y)+\psi(Y),
\end{equation}
where $ \varphi(Y) $ and $ \psi(Y) $ are respectively defined as follows,
\begin{align}
	\varphi(Y)&=\Bigg\{\begin{aligned}
		&-R^{-1} \nabla \tilde{c}\left(y^{r}\right)-\beta_{1} \mathcal{L} y^{r}-\beta_{2} z, \\
		&\beta_{1} \mathcal{L} y^{r}, \quad z(0)=0,
	\end{aligned}\\
	\psi(Y)&=\Bigg\{\begin{aligned}
		&(R^{-1}-\Xi^{-1}) \nabla \tilde{c}\left(y^{r}\right), \\
		&~~\mathbf{0}.
	\end{aligned}	
\end{align}

The rest of the proof can be accomplished by the following two steps.

\textbf{Step 1.} Establish the exponential stability of system $ \dot{Y}=\varphi(Y) $. Let $ \bar{Y}=\operatorname{col}(\bar{y}^{r}, \bar{z}) $ denote the equilibrium point of system $ \dot{Y}=\varphi(Y) $. In what follows, we first reveal the relationship between  $ \bar{y}^{r} $ and the optimal solution $ s^{\star} $. Note that the dynamics of $ \bar{Y} $ satisfy
\begin{subequations}\label{equilibrium point_r}
	\begin{align}
	\mathbf{0} = & -R^{-1}\nabla \tilde{c}(\bar{y}^{r}) - \beta_{1}\mathcal{L}\bar{y}^{r} - \beta_{2}\bar{z}, \label{equilibrium point_r_a}\\
	\mathbf{0} = & \beta_{1}\mathcal{L}\bar{y}^{r}.\label{equilibrium point_r_b}
	\end{align}
\end{subequations}

It follows from (\ref{equilibrium point_r_b}) that the vector $ \bar{y}^{r} $ resides in the null-space of $ \mathcal{L} $. Therefore, it can be claimed that $ \bar{y}^{r}=1_{N}\otimes \varsigma $ holds for some vector $ \varsigma \in \mathbb{R} $. Then pre-multiplying both sides of (\ref{compact form_r_b}) by $ \varrho^{\mathrm{T}} $ yields $  \varrho^{\mathrm{T}} \dot{z}= 0 $, which implies that $ \varrho^{\mathrm{T}} \bar{z}= 0 $ under the assumption $ z(0)= \mathbf{0} $. Then, pre-multiplying both sides of (\ref{equilibrium point_r_a}) by $ \varrho^{\mathrm{T}} $ results in $  \mathbf{1}_{N}^{\mathrm{T}}\nabla\tilde{c}(\bar{y}^{r})= \mathbf{0} $. Together with $ \bar{y}^{r}=1_{N}\otimes \varsigma $, one has $ \sum_{i=1}^{N}\nabla c_{i}(\varsigma)=\mathbf{0} $. Note that the optimality condition is $ \sum_{i=1}^{N}\nabla c_{i}(s^{\star})=\mathbf{0} $. One then has $ \bar{y}^{r}=1_{N}\otimes s^{\star} $. It is thus proved that $ \bar{y}_{i}^{r} $ is the same as the optimal solution $ s^{\star} $ of the global cost function.

Now we are ready to show that the equilibrium point $ \bar{Y}=\operatorname{col}(\bar{y}^{r}, \bar{z}) $ of system $ \dot{Y}=\varphi(Y) $ is exponentically stable. To facilitate the analysis, introduce the coordinate transformations $ \tilde{y}^{r}= y^{r}-\bar{y}^{r} $ and $ \tilde{z} = z-\bar{z} $. Then the dynamics of $ \tilde{y}^{r} $ and $ \tilde{z} $ are given as follows,
\begin{subequations} \label{origin dynamics_r}
	\begin{align}
	\dot{\tilde{y}}^{r}= & -R^{-1} \Delta_{c}\left(\tilde{y}^{r}\right)-\beta_{1} \mathcal{L} \tilde{y}^{r}-\beta_{2} \tilde{z}, \label{origin dynamics_r_a}\\
	\dot{\tilde{z}}= & \beta_{1} \mathcal{L} \tilde{y}^{r}, \label{origin dynamics_r_b}
	\end{align}
\end{subequations}
where $ \Delta_{c}\left(\tilde{y}^{r}\right)=\nabla \tilde{c}\left(\tilde{y}^{r}+\bar{y}^{r}\right)-\nabla \tilde{c}\left(\bar{y}^{r}\right) $. Accordingly, the equilibrium point of the dynamics (\ref{origin dynamics_r}) is transferred to the origin. As a consequence, to show the exponential stability of the equilibrium point $ \bar{Y} $, it is sufficient to prove that the origin is an exponentially stable equilibrium point of the dynamics (\ref{origin dynamics_r}).

To this end, consider the Lyapunov function candidate $ V_{0}\left(\tilde{y}^{r}, \tilde{z}\right) =\frac{1}{2} \tilde{y}^{r \operatorname{T}} R \tilde{y}^{r}+\frac{1}{2}\left(\tilde{y}^{r}+\tilde{z}\right)^{\operatorname{T}} R\left(\tilde{y}^{r}+\tilde{z}\right)$. Then the derivative of $ V_{0} $ along the trajectories of dynamics (\ref{origin dynamics_r}) satisfies
\begin{align} \label{dot_V_0_1}
\dot{V}_{0} =& -2 \tilde{y}^{r \operatorname{T}}\Delta_{c}\left(\tilde{y}^{r}\right) - \beta_{1}\tilde{y}^{r \operatorname{T}}R\mathcal{L}\tilde{y}^{r} - 2\beta_{2}\tilde{y}^{r \operatorname{T}}R\tilde{z} \notag\\
& - \tilde{z}^{\operatorname{T}}\Delta_{c}\left(\tilde{y}^{r}\right) - \beta_{2}\tilde{z}^{\operatorname{T}}R\tilde{z}.
\end{align}

Let $ \varpi=\min(\varpi_{1},\varpi_{2},\ldots,\varpi_{N}) $ be the smallest strongly convex coefficient of the cost functions $ c_{i} $'s, $ \bar{\iota}=\max(\iota_{1},\iota_{2},\ldots,\iota_{N}) $ be the largest Lipschitz coefficient of the gradients $ \nabla c_{i} $'s, and $ \underline{\varrho}= \min(\varrho_{1},\varrho_{2},\ldots,\varrho_{N}) $ be the smallest component of the left eigenvector $ \varrho $. Under Assumption \ref{cost function assumption}, one has $ \tilde{y}^{r \operatorname{T}}\Delta_{c}\left(\tilde{y}^{r}\right) \geq \varpi \left\|\tilde{y}^{r}\right\|^{2} $ and $ \left\|\Delta_{c}\left(\tilde{y}^{r}\right)\right\| \leq \bar{\iota}\left\|\tilde{y}^{r}\right\| $.  Thus, the following inequality is satisfied,
\begin{equation*}
- \tilde{z}^{\operatorname{T}}\Delta_{c}\left(\tilde{y}^{r}\right)  \leq \dfrac{1}{4\delta}\|\Delta_{c}\left(\tilde{y}^{r}\right)\|^{2} + \delta\|\tilde{z}\|^{2} \leq \dfrac{\bar{\iota}^{2}}{4\delta} \|\tilde{y}^{r}\|^{2} + \delta\|\tilde{z}\|^{2},
\end{equation*}
where $ \delta $ is any positive constant. Moreover, it follows from Lemma \ref{graph theory lemma} that $ \tilde{y}^{r \operatorname{T}}R\mathcal{L}\tilde{y}^{r} = \tilde{y}^{r \operatorname{T}}\bar{\mathcal{L}}\tilde{y}^{r} $ and $ \tilde{z}^{\operatorname{T}}R\tilde{z}\geq \underline{\varrho} \|\tilde{z}\|^{2} $. 
With these results, (\ref{dot_V_0_1}) can be rewritten as follows,
\begin{align}\label{dot_V_0_2}
\dot{V}_{0} \leq & -\Big(2 \varpi- \dfrac{\bar{\iota}^{2}}{4\delta} \Big)\left\|\tilde{y}^{r}\right\|^{2} - (\beta_{2}\underline{\varrho}-\delta ) \|\tilde{z}\|^{2}  \notag\\
& - \beta_{1}\tilde{y}^{r \operatorname{T}}\bar{\mathcal{L}}\tilde{y}^{r} - 2\beta_{2}\tilde{y}^{r \operatorname{T}}R\tilde{z}.
\end{align}

It then follows from \romannumeral2) of Lemma \ref{graph theory lemma} that there exist orthogonal vectors $ \mathbf{1}_{N} $ and $ \zeta_{i},~i=2,3,\ldots,N $, such that $ \bar{\mathcal{L}}\mathbf{1}_{N}=\mathbf{0} $ and $ \bar{\mathcal{L}}\zeta_{i}=\lambda_{i}\zeta_{i} $. Define the matrix $ \varPhi=\left(\mathbf{1}_{N},\zeta_{2},\zeta_{3},\ldots,\zeta_{N} \right)\in \mathbb{R}^{N\times N} $. Let $ \theta=\operatorname{col}\left( \theta_{1},\theta_{2},\ldots,\theta_{N} \right)\in \mathbb{R}^{N} $ with constants $ \theta_{1}= \mathbf{1}_{N}^{\operatorname{T}} \tilde{y}^{r} $ and $ \theta_{i}= \zeta_{i}^{\operatorname{T}} \tilde{y}^{r} $, $i=2,3,\ldots,N $.  It then can be verified that $ \tilde{y}^{r} = \varPhi \theta $. Therefore, one can further deduce that  
\begin{align} \label{inequality_1.1}
\tilde{y}^{r T}\bar{\mathcal{L}}\tilde{y}^{r} 
&=\operatorname{tr}\big(\theta^{\mathrm{T}} \varPhi^{\mathrm{T}} \bar{\mathcal{L}} \varPhi \theta\big) =\operatorname{tr}\big( \varPhi^{\mathrm{T}} \bar{\mathcal{L}} \varPhi \theta \theta^{\mathrm{T}}\big)  \notag\\
&=\sum_{i=2}^{N} \zeta_{i}^{\mathrm{T}} \bar{\mathcal{L}} \zeta_{i} \theta_{i}^{\mathrm{T}} \theta_{i} = \sum_{i=2}^{N} \lambda_{i} \zeta_{i}^{\mathrm{T}} \zeta_{i} \theta_{i}^{\mathrm{T}} \theta_{i} \notag\\[-1mm]
& \geq \lambda_{2} \sum_{i=2}^{N} \theta_{i}^{\mathrm{T}} \theta_{i} =\lambda_{2}\|\nu\|^{2},
\end{align}
where $ \nu=\operatorname{col}(\theta_{2},\theta_{3},\ldots,\theta_{N}) $. Similarly, one has
\begin{align} \label{y_xi_z_1}
\tilde{y}^{r \operatorname{T}}R\tilde{z} 
&=\tilde{z}^{\operatorname{T}}R\tilde{y}^{r}  = \tilde{z}^{\operatorname{T}} R \varPhi \theta  \notag\\
&=\tilde{z}^{\operatorname{T}} R \mathbf{1}_{N}\theta_{1}  + \tilde{z}^{\operatorname{T}} R \varPsi,
\end{align}
where $ \varPsi = \sum_{i=2}^{N} \zeta_{i}\theta_{i} $. Since $ \zeta_{i}$, $i=2,3,\ldots,N $ are orthogonal vectors, it then can be obtained that $ \| \varPsi  \|^{2} = \|\nu\|^{2} $. Note that $ \tilde{z}^{\operatorname{T}} R \mathbf{1}_{N}= \varrho^{\mathrm{T}} \tilde{z} = \mathbf{0} $. It then follows from (\ref{y_xi_z_1}) and $ \varrho_{i} < 1 $, $ i=1,2,\ldots,N $ that for any positive constant $ \delta $, 
\begin{align}\label{inequality_1.2}
- 2\beta_{2}\tilde{y}^{r \operatorname{T}}R\tilde{z} & \leq \dfrac{\beta_{2}^{2}}{\delta} \| \varPsi  \|^{2} + \delta \| \tilde{z}\|^{2}  \notag\\
& = \dfrac{\beta_{2}^{2}}{\delta} \| \nu \|^{2} + \delta \| \tilde{z}\|^{2}.
\end{align}

Substituting (\ref{inequality_1.1}) and (\ref{inequality_1.2}) into (\ref{dot_V_0_2}) leads to
\begin{equation*}
\begin{aligned}
\dot{V}_{0} \leq & -\left(2 \varpi-\frac{\bar{\iota}^{2}}{4\delta}\right)\left\|\tilde{y}^{r}\right\|^{2}-\left(\beta_{2} \underline{\varrho}-2\delta\right)\|\tilde{z}\|^{2} \\
& -\left(\beta_{1} \lambda_{2}-\frac{\beta_{2}^{2}}{\delta}\right)\|\nu\|^{2}.
\end{aligned}
\end{equation*}
Then, successively choose the constants $ \delta $, $ \beta_{2} $ and $ \beta_{1} $ such that the following inequalities are satisfied,
\begin{equation}\label{paremeters_beta_i}
	2 \varpi-\dfrac{\bar{\iota}^{2}}{4\delta} > 0,\quad \beta_{2} \underline{\varrho}-2\delta > 0, \quad \beta_{1} \lambda_{2}-\dfrac{\beta_{2}^{2}}{\delta} > 0.
\end{equation}
One thus has 
\begin{equation} \label{dot_V_0_final}
\dot{V}_{0} \leq-\mu_{0}\left(\left\|\tilde{y}^{r}\right\|^{2}+\|\tilde{z}\|^{2}\right),
\end{equation}
where $ \mu_{0} = \min\big\{2 \varpi-\frac{\bar{\iota}^{2}}{4\delta},~ \beta_{2} \underline{\varrho}-2\delta,~ 1 \big\} $. It is noted that $ V_{0}(\tilde{y}^{r}, \tilde{z}) $ can be rewritten as
$
V_{0}=\left(\begin{array}{c}
	\tilde{y}^{r} \\
	\tilde{z} 
\end{array}\right)^{\mathrm{T}} F\left(\begin{array}{c}
	\tilde{y}^{r} \\
	\tilde{z} 
\end{array}\right)	
$ with
$
F=\dfrac{1}{2}\left(\begin{array}{cc}
	2 & 1  \\
	1 & 1
\end{array}\right)\otimes R	
$. Let $ \varepsilon $ denote the maximum eigenvalue of $ F $ . One then has $ V_{0}\leq \varepsilon\big(\left\|\tilde{y}^{r}\right\|^{2}+\|\tilde{z}\|^{2}\big) $. Thus, (\ref{dot_V_0_final}) can be rewritten as $ \dot{V}_{0} \leq -\frac{\mu_{0}}{\varepsilon}V_{0} $. It can be concluded that the origin is an exponentially stable equilibrium point of the dynamics (\ref{origin dynamics_r}). This indicates that the equilibrium point $ (\bar{y}^{r}, \bar{z}) $ of system $ \dot{Y}=\varphi(Y) $ is exponentially stable. \\[-2mm]

\textbf{Step 2.} Establish the exponential stability of system $ \dot{Y}=\varphi(Y)+\psi(Y) $. To this end, we rewrite the perturbation term $ \psi(Y) $ as $ \psi(Y)=\psi_{1}(Y)+\psi_{2}(t) $, where
\begin{align}
	\psi_{1}(Y)&=\Bigg\{\begin{aligned}
		&(R^{-1}-\Xi^{-1}) \big(\nabla \tilde{c}(y^{r})-\nabla \tilde{c}(\bar{y}^{r})\big), \\
		&~~\mathbf{0},
	\end{aligned}\\
	\psi_{2}(t)&=\Bigg\{\begin{aligned}
	&(R^{-1}-\Xi^{-1})\nabla \tilde{c}(\bar{y}^{r}), \\
	&~~\mathbf{0}.
\end{aligned}	
\end{align}  

Note that the perturbation term $ \psi_{1}(Y) $ satisfies $ \psi_{1}(\bar{Y})= \mathbf{0} $ and $ \psi_{1}(Y) \leq \sigma_{1}(t) \|Y-\bar{Y}\| $ with $ \sigma_{1}(t)= \bar{\iota} \max_{i}| \varrho_{i}^{-1}-(\xi_{i}^{i}(t))^{-1} | $. Since it is proved that $ \lim_{t\to\infty} \xi_{i}^{i}(t) = \varrho_{i} $ exponentially, one has $ \lim_{t\to\infty} \sigma_{1}(t)= 0 $ exponentially. Then it follows from Lemma \ref{lemma 4} that the equilibrium point $ \bar{Y} $ of the perturbed system $ \dot{Y}=\varphi(Y)+\psi_{1}(Y) $ is exponentially stable. 

Furthermore, since $ \nabla\tilde{c} $ is globally Lipschitz by Assumption \ref{cost function assumption}, we learn that $  \dot{Y}=\varphi(Y)+\psi_{1}(Y)+\psi_{2}(t) $ is globally Lipschitz in $ Y $. The boundness of $ \nabla \tilde{c}(\bar{y}^{r}) $ suggests that $ \psi_{2}(t) $ is bounded. Then it follows from Lemma 4.6 in \cite{khalil2002nonlinear} that the system $  \dot{Y}=\varphi(Y)+\psi_{1}(Y)+\psi_{2}(t) $ is input-to-state stable (ISS). Note  also that $ \lim_{t\to\infty} \psi_{2}(t) = \mathbf{0} $ exponentially. It can be shown that $ Y $ exponentially converges to $ \bar{Y} $ by the property of ISS given in \cite{khalil2002nonlinear}. Therefore, we obtain that $ y^{r} $ exponentially converges to $ \bar{y}^{r}= \mathbf{1}_{N}\otimes s^{*} $, with $ s^{*} $ being the optimal solution to minimize the global cost function. The proof of Theorem \ref{proposition1} is thus completed.\qed

\section{Proof of Lemma \ref{lemma_chi_i}} \label{Appendix_2}

First, choose the positive definite function $ V_{i\bar{x}}=\bar{x}_{i 1}^{2} $. Then the derivative of $ V_{i\bar{x}} $ along the trajectories of dynamics (\ref{augmented system 2}) is given as follows,
\begin{align} \label{dot_V_ix_1}
\dot{V}_{i\bar{x}} = & 2 \bar{x}_{i 1} \dot{\bar{x}}_{i 1}= -2 \gamma\bar{x}_{i 1}^{2}+2 \bar{x}_{i 1} \vartheta_{i}-2 \bar{x}_{i 1} \dot{y}_{i}^{r} \notag\\
\leq & -2(\gamma-1)\bar{x}_{i 1}^{2}+\vartheta_{i}^{2}+\big|\dot{y}_{i}^{r}\big|^{2}.
\end{align}
Letting $ \gamma \geq \frac{3}{2} $ yields
\begin{align*}
\dot{V}_{i\bar{x}} \leq -\bar{x}_{i 1}^{2}+\vartheta_{i}^{2}+\big|\dot{y}_{i}^{r}\big|^{2}.
\end{align*}
Then by applying the changing supply function technique in \cite{sontag1995changing}, given any smooth function $ \Delta_{i\bar{x}}(\bar{x}_{i 1})>0 $, there exists a continuously differentiable function $\bar{V}_{i\bar{x}}(\bar{x}_{i 1})$ such that, along the trajectories of (\ref{augmented system 2}), the following inequalities are satisfied,
\begin{align}
&\underline{\alpha}_{i\bar{x}}(\|\bar{x}_{i 1}\|) \leq \bar{V}_{i\bar{x}}(\bar{x}_{i 1}) \leq \overline{\alpha}_{i\bar{x}}(\|\bar{x}_{i 1}\|), \\
& \dot{\bar{V}}_{i\bar{x}} \leq -\Delta_{i\bar{x}}(\bar{x}_{i 1})\bar{x}_{i 1}^{2} + \phi_{i\vartheta}\left(\vartheta_{i}\right)\vartheta_{i}^{2} + \phi_{ir}\left(\dot{y}_{i}^{r}\right)\big|\dot{y}_{i}^{r}\big|^{2}, \label{dot_bar_V_bar_x_i1}
\end{align}
for some known smooth functions $ \underline{\alpha}_{i\bar{x}}, \overline{\alpha}_{i\bar{x}}\in \mathcal{K}_{\infty} $ and $ \phi_{i\vartheta}, \phi_{ir} >1 $.

Second, since $ M_{i} $ is Hurwitz, there exists a unique positive definite matrix $ P_{i} $ satisfying $P_{i} M_{i}+M_{i}^{\operatorname{T}} P_{i}=-I$. Consider the positive definite function $V_{i\tilde{\eta}}=2 \tilde{\eta}_{i}^{\operatorname{T}} P_{i} \tilde{\eta}_{i}$. Its derivative along the trajectories of (\ref{augmented system 2}) is given as
\begin{align} \label{dot_V_i_tilde_eta_1}
\dot{V}_{i\tilde{\eta}} = & -2\left\|\tilde{\eta}_{i}\right\|^{2}+4 \tilde{\eta}_{i}^{\operatorname{T}} P_{i} \tilde{f}_{i 1}-4 b_{i}^{-1} \tilde{\eta}_{i}^{\operatorname{T}} P_{i} N_{i} \epsilon_{i} \notag\\
\leq &  -\left\|\tilde{\eta}_{i}\right\|^{2}+8\left\|P_{i}\right\|^{2}\big\|\tilde{f}_{i 1}\big\|^{2}\!+\!8b_{i}^{-2}\left\|P_{i} N_{i}\right\|^{2}\|\epsilon_{i}\|^{2}.
\end{align}

Note that $ \tilde{f}_{i 1}\left(\bar{x}_{i 1}, \vartheta_{i}, y_{i}^{r}, d\right) $ is sufficiently smooth and satisfies $\tilde{f}_{i 1}\left(0,0, y_{i}^{r}, d\right)=0$. By Lemma 11.1 in \cite{chen2015stabilization}, there exist known smooth functions $ \bar{\phi}_{i\bar{x}}, \bar{\phi}_{i\vartheta} >1$ and unknown constant $ \bar{\gamma}_{i1}>1 $ such that, for all $ y_{i}^{r}\in\mathbb{R} $ and $ d\in\mathbb{D} $,  the following inequality holds,
\begin{align*}
\big\|\tilde{f}_{i 1}\big\|^{2} \leq \bar{\gamma}_{i 1} \left(\bar{\phi}_{i\bar{x}}\left(\bar{x}_{i 1}\right)\bar{x}_{i 1}^{2}+\bar{\phi}_{i\vartheta}\left(\vartheta_{i}\right)\vartheta_{i}^{2}\right).
\end{align*}
Therefore, (\ref{dot_V_i_tilde_eta_1}) can be rewritten as follows,
\begin{align} \label{dot_V_tilde_eta_i}
\dot{V}_{i\tilde{\eta}} \leq & -\left\|\tilde{\eta}_{i}\right\|^{2} +8 \bar{\gamma}_{i 1}\left\|P_{i}\right\|^{2} \bar{\phi}_{i\bar{x}}\left(\bar{x}_{i 1}\right) \bar{x}_{i 1}^{2} \notag\\
&+8 \bar{\gamma}_{i 1}\left\|P_{i}\right\|^{2} \bar{\phi}_{i\vartheta}\left(\vartheta_{i}\right) \vartheta_{i}^{2} +8b_{i}^{-2}\left\|P_{i} N_{i}\right\|^{2}\|\epsilon_{i}\|^{2}.
\end{align}

Last, consider the positive definite function $V_{i 1}(\chi_{i})=\hbar_{i} \bar{V}_{i\bar{x}}(\bar{x}_{i 1})+V_{i\tilde{\eta}}({\eta}_{i})$, where $ \hbar_{i} $ is a positive constant to be determined later. Then there exist known smooth functions $ \underline{\alpha}_{i\chi}, \overline{\alpha}_{i\chi}\in \mathcal{K}_{\infty} $ such that $ \underline{\alpha}_{i\chi}(\|\chi_{i}\|) \leq V_{i1}(\chi_{i}) \leq \overline{\alpha}_{i\chi}(\|\chi_{i}\|) $ is satisfied. Moreover, by combining (\ref{dot_bar_V_bar_x_i1}) and (\ref{dot_V_tilde_eta_i}), the derivative of $ V_{i 1} $ along the trajectories of dynamics (\ref{augmented system 2}) satisfies
\begin{align} \label{dot_V_i1_1}
\dot{V}_{i 1}\leq &-\left(\hbar_{i} \Delta_{i\bar{x}}\left(\bar{x}_{i 1}\right)-8 \bar{\gamma}_{i 1}\left\|P_{i}\right\|^{2} \bar{\phi}_{i\bar{x}}\left(\bar{x}_{i 1}\right)\right)\bar{x}_{i 1}^{2}-\left\|\tilde{\eta}_{i}\right\|^{2} \notag\\
&+\left(\hbar_{i} \phi_{i\vartheta}\left(\vartheta_{i}\right)+8 \bar{\gamma}_{i 1}\left\|P_{i}\right\|^{2} \bar{\phi}_{i\vartheta}\left(\vartheta_{i}\right)\right)\vartheta_{i}^{2}\notag\\
&+\hbar_{i} \phi_{ir}\left(\dot{y}_{i}^{r}\right)\left|\dot{y}_{i}^{r}\right|^{2}+8b_{i}^{-2}\left\|P_{i} N_{i}\right\|^{2}|\epsilon_{i}|^{2}.
\end{align}
Let $ \hbar_{i}\geq 8 \bar{\gamma}_{i 1}\left\|P_{i}\right\|^{2}+1 $, $ \Delta_{i\bar{x}}\left(\bar{x}_{i 1}\right)\geq \bar{\phi}_{i\bar{x}}\left(\bar{x}_{i 1}\right)+1 $, $\gamma_{i\vartheta} \geq \hbar_{i} +8 \bar{\gamma}_{i 1}\left\|P_{i}\right\|^{2}$, $\hat{\phi}_{i\vartheta}\left(\vartheta_{i}\right) \geq \phi_{i\vartheta}\left(\vartheta_{i}\right)+\bar{\phi}_{i\vartheta}\left(\vartheta_{i}\right)$, $ \gamma_{ir}\geq \hbar_{i}+1 $, $ \hat{\phi}_{ir}\left(\dot{y}_{i}^{r}\right) \geq \phi_{ir}\left(\dot{y}_{i}^{r}\right) $, $ \gamma_{i\epsilon}\geq 8b_{i}^{-2}\left\|P_{i} N_{i}\right\|^{2} + 1 $. Then (\ref{dot_V_i1_1}) can be rewritten as
\begin{align*} \label{dot_V_i1_2}
\dot{V}_{i1} \leq & -\chi_{i}^{2} +\gamma_{i\vartheta} \hat{\phi}_{i\vartheta}\left(\vartheta_{i}\right) \vartheta_{i}^{2} +\gamma_{ir} \hat{\phi}_{ir}\left(\dot{y}_{i}^{r}\right)\left|\dot{y}_{i}^{r}\right|^{2}  +\gamma_{i\epsilon}|\epsilon_{i}|^{2}.
\end{align*}

The proof of Lemma \ref{lemma_chi_i} is thus completed.\qed

\end{appendices}

\bibliographystyle{IEEEtran}  
\bibliography{IEEEabrv,mylib}

\end{document}